\documentclass[a4paper,11pt]{article}

\usepackage[english]{babel}
\usepackage{amsmath, amssymb , latexsym, amsthm,
  epic, epsfig, rotating}
\usepackage{xypic}

\newtheorem{definition}{Definition}[section]
\newtheorem{thm}[definition]{Theorem}

\newtheorem{lemm}[definition]{Lemma}

\newtheorem{prop}[definition]{Proposition}
\newtheorem{adde}[definition]{Addendum}
\newtheorem{claim}[definition]{Claim}

\newtheorem{bem}[definition]{Remark}

\numberwithin{equation}{section}

\def\inte{\mathrm{Int}}

\title{\textsc{Denjoy constructions for fibered homeomorphisms of the torus}}
\author{Fran\c cois B\'eguin, Sylvain Crovisier,
  Tobias J\"ager and Fr\'ed\'eric Le Roux}

\newcommand{\Gammatil}{\ensuremath{\tilde{\Gamma}}}

\newcommand{\ldot}{\ensuremath{\textbf{.}}}
\newcommand{\ld}{\ensuremath{,\ldots,}}
\newcommand{\ssq}{\ensuremath{\subseteq}}
\def\smin{\ensuremath{\setminus}}
\newcommand{\ra}{\ensuremath{\rightarrow}}

\newcommand{\N}{\ensuremath{\mathbb{N}}} 
\newcommand{\R}{\ensuremath{\mathbb{R}}}
\newcommand{\Z}{\ensuremath{\mathbb{Z}}}
\newcommand{\Q}{\ensuremath{\mathbb{Q}}}

\newcommand{\kreis}{\ensuremath{\mathbb{T}^{1}}}

\newcommand{\alphlist}{\begin{list}{(\alph{enumi})}{\usecounter{enumi}}}
\newcommand{\romanlist}{\begin{list}{(\roman{enumi})}{\usecounter{enumi}}}
\newcommand{\listend}{\end{list}}

\newcommand{\thx}{\ensuremath{(\theta,x)}}
\newcommand{\thom}{\ensuremath{\theta + \omega}}
\newcommand{\ntorus}[1][2]{\ensuremath{\mathbb{T}^{#1}}}

\newcommand{\NN}{\mathbb{N}}

\newcommand{\PP}{\mathbb{P}}

\newcommand{\RR}{\mathbb{R}}

\newcommand{\TT}{\mathbb{T}}

\newcommand{\ZZ}{\mathbb{Z}}

\newcommand{\Leb}{\ensuremath{\mathrm{Leb}}}
\def\De{{\Delta}}

\def\adhe{\mathrm{Adh}}

\begin{document}
\maketitle

\sloppy

\abstract{We construct different types of quasiperiodically forced circle
homeomorphisms with transitive but non-minimal dynamics. Concerning the recent Poincar\'e-like classification for this class of maps
of~\cite{jaeger/stark:2006}, we demonstrate that transitive but non-minimal behaviour can occur in each of the different cases. This closes one of the last gaps in the topological classification. 
 
 Actually, we are able to get some transitive quasiperiodically forced circle homeomorphisms with rather complicated minimal sets. For example, we show that, in some of the examples we construct, the unique minimal set is a Cantor set and its intersection with each vertical fibre is uncountable and nowhere dense (but may contain isolated points). 
 
We also prove that minimal sets of the later kind cannot occur when the dynamics are given by the projective action of a quasiperiodic $\mbox{SL}(2,\R)$-cocycle. More precisely, we show that, for a quasiperiodic $\mbox{SL}(2,\R)$-cocycle, any minimal strict subset of the torus either is a union of finitely many continuous curves, or contains at most two points on generic fibres.}

\bigskip

\section{Introduction}
We study homeomorphisms of the two-torus which are isotopic to the identity, and of the form
\begin{equation}
\begin{array}{cccccc} f & : & \ntorus \ra \ntorus & , & \thx \mapsto
  (\thom,f_\theta(x)) \ , \end{array} 
\end{equation}
with $\omega \in \R\smin\Q$. 
Such homeomorphisms are often called \emph{quasi-periodically forced (qpf) circle homeomorphisms}, their class  will be denoted  by ${\cal F}$.

Skew products like this occur in various situations in physics. One well-known
example is the so-called Harper map, which appears in the study of
quasi-crystals and the corresponding Schr\"odinger operators (see, for
example, \cite{aubry/andre:1980,haro/puig:2006}). Another one is the qpf
Arnold circle map, which is used as a simple model for oscillators forced with
two or more incommensurate frequencies \cite{ding/grebogi/ott:1989}.  

\bigskip

The interest in transitive but non-minimal dynamics in this kind of maps is
motivated by a recent classification result in \cite{jaeger/stark:2006}, which
we briefly want to discuss in order to motivate the problem.

Given any lift $F:\kreis \times\R \to \kreis \times \R$, $\theta\in\kreis$ and
$n\in\N$, let $$F^n_\theta := F_{\theta+(n-1)\omega} \circ \ldots \circ
F_\theta\ .$$ Then the limit
\begin{equation} \label{e.rotnum}
  \rho(F) \ = \ \lim_{n\to\infty} (F_\theta^n(x)-x)/n
\end{equation}
exists and does not depend on $\thx \in \kreis \times \R$. Furthermore, the
convergence in (\ref{e.rotnum}) is uniform \cite{herman:1983}. The angle $\rho(f) = \rho(F) \bmod 1$ is called the {\em (fibered) rotation number} of $f$. However, unlike the
one-dimensional case, the {\em deviations from the average rotation}, given
by 
\begin{equation}
  \label{e.deviations}
D_n\thx \ := \ F_\theta^n(x)-x-n\rho(F) \ ,
\end{equation}
need not be uniformly bounded in $n,\theta,x$ anymore.%
\footnote{In the case of an unforced circle homeomorphism, the uniform bound
  is 1.} 
This gives rise to a basic dichotomy: a homeomorphism $f\in{\cal F}$ is called {\em
  $\rho$-bounded} if $\sup_{n,\theta,x}|D_n\thx| < \infty$ and {\em
  $\rho$-unbounded} otherwise.

Another concept which turned out to be fundamental in this context are
\emph{\mbox{($p,q$)-invariant} strips}. These are compact invariant sets which
intersect each vertical fibre $\{\theta\} \times\kreis$ in
exactly $pq$ compact intervals, have an internal $p$-periodic structure and
certain additional regularity properties. Since the precise formulation is
somewhat technical and we will not use it later, we refrain from stating it
here and refer to \cite{jaeger/stark:2006} or \cite{jaeger/keller:2006} for
the definition.  Among \mbox{($p,q$)-invariant} strips are the
\emph{continuous \mbox{($p,q$)-invariant} graphs}.  These are the minimal
invariant subsets of $\TT^2$ on which  the restriction of $p_{1}: (\theta,x)
\mapsto \theta$ is a $pq$-fold covering with $p$ connected components. In
order to have a rough idea, the reader should just think of an invariant strip
as a generalisation of a continuous \mbox{($p,q$)-invariant} graph, where each
point of the graph is possibly replaced by a vertical segment.  In particular,
the existence of such an object forces the rotation number $\rho(f)$ to be
rationally related to $\omega$ and the deviations (\ref{e.deviations}) to be
bounded \cite{jaeger/keller:2006}.

It turns out that in the $\rho$-bounded case a direct analogue to the
Poincar\'e Classification Theorem (e.g.\ \cite{katok/hasselblatt:1997}) holds,
with invariant strips playing the role of periodic orbits in the unforced
case:

\begin{thm}[theorems 3.1 and 4.1 in \cite{jaeger/stark:2006}]
  \label{thm:poincare}~
  \alphlist
  \item If $f\in {\cal F}$ is $\rho$-bounded, then either there exists a
  $(p,q)$-invariant strip and $\rho(f)$, $\omega$ and 1 are rationally
  dependent or $f$ is semi-conjugate to the irrational torus
  translation $\thx \mapsto (\thom,x+\rho(f))$ by a continuous semi-conjugacy $h$
  which is fibre-respecting (i.e.\ $p_1\circ h = p_1$).
  \item If $f\in {\cal F}$ is $\rho$-unbounded, then it is topologically
 transitive.  \listend 
 \end{thm}
 
\bigskip
 
Since all known examples of $\rho$-unbounded
behaviour are either minimal (all $\rho$-unbounded skew rotations are
 minimal, see \cite[proposition 4.2.6]{katok/hasselblatt:1997}; other examples
 are given in \cite{bjerkloev:2005a}) or their topological dynamics have not
 yet been clarified, this immediately raises the question whether transitive
 but non-minimal dynamics can occur in the $\rho$-unbounded case. Similarly,
 it is not known whether this is possible when $f$ is semi-conjugate to an
 irrational rotation - which could be interpreted as ``a Denjoy
 counter-example without wandering sets''.%
\footnote{In this context, we would also like to mention the constructions by
Mary Rees in \cite{rees:1979} and \cite{rees:1981} (see also
\cite{becroler:2006}).  These equally produce Denjoy-like examples without
wandering sets on the two-torus, but in the fibered case the dynamics will
always remain minimal.}

The positive answer to these questions is provided by the following theorem,
which is the main result of this paper:
\begin{thm}~

\begin{itemize}
\item There exists a transitive non minimal  qpf circle homeomorphism which is 
$\rho$-unbounded;
\item There exists a transitive non minimal  qpf circle homeomorphism which is 
semi-conjugate to an irrational rotation.
\end{itemize}
\end{thm}
This theorem will follow from a general construction.
\begin{thm} \label{t.mainthm}
  Suppose $R$ is a minimal qpf circle homeomorphism. Then
  there exist a continuous and surjective map $\pi : \ntorus \ra \ntorus$ and
  a qpf circle homeomorphism $f$ such that $\pi \circ f =
  R \circ \pi$ and $f$ is topologically transitive, but not minimal.  In
  addition, if $R$
  is a diffeomorphism,%
\footnote{In fact, it suffices that all fibre maps $R_\theta$ are circle
  diffeomorphisms and $\partial_x R_\theta$ depends continuously on \thx.}
 then $f$ can be chosen such that all fibre maps $f_\theta$ are circle
  diffeomorphisms and $\partial_x f_\theta$ depends continuously on $\thx$.
\end{thm}

This can be interpreted as follows. If $R$ is a minimal qpf circle
homeomorphism with certain additional properties, and if these properties are
preserved by topological extension, then there exists $f\in{\cal F}$ with the
same properties, but transitive and non-minimal dynamics. In particular, this
is true for the properties {\em `semi-conjugate to an irrational rotation'}
and {\em `unbounded deviations'}. Another such property, related to the
structure of the ergodic invariant measures, will be discussed in
section~\ref{ErgodicMeasures}~.  \medskip

The proof of theorem \ref{t.mainthm} will be given via propositions
\ref{prop.construction-curve} and \ref{p.blow-up-curve} below and their addenda,
which immediately imply the above statement. The construction we carry out is
very similar to Denjoy's construction of circle homeomorphisms with wandering
sets. The main idea is to start with a continuous curve $\Gamma$ and to `blow
up' this curve and all its images to small annuli, just as the points of an
orbit are blown up to wandering intervals in Denjoy's construction. However,
instead of requiring that $\Gamma$ is disjoint from all its images, which
would lead to wandering sets, we choose the curve $\Gamma$ such that there are
`many' intersections, and this fact is then used to establish the transitivity
of $f$. Further, it turns out that in order to make the construction work, the
initial curve must have another, rather surprising property: whenever it
intersects any of its images this must happen over a whole interval - in other
words the connected components of the intersection must not be singletons (see
definition~\ref{def:flatintersections} and
proposition~\ref{prop.construction-curve}). This property will turn out to be
crucial in order to ensure the continuity of the semi-conjugacy $\pi$ during
the construction.

The construction of such a curve $\Gamma$ is first carried out in the case
where $R$ is real-analytic, since this allows to avoid some technical problems
and renders the main ideas more visible.

It should be mentioned that there exist well-known examples of qpf circle
homeomorphisms with transitive but non-minimal dynamics, which are due to
Shnirelman \cite{shnirelman:1930} (see also, for example, \cite[section
12.6(b)]{katok/hasselblatt:1997}).  However, in these examples there always
exists an invariant curve, which is just a special case of an invariant
strip. Consequently, the resulting minimal set (the invariant curve) has have
a very simple structure. In contrast to this, it is known that minimal strict
subsets of \ntorus\ in the absence of invariant strips must be much more
complicated
(see proposition~\ref{prop.jeakel} below, taken from \cite[theorem 4.5 and lemma 4.6]{jaeger/keller:2006}). In
particular, we obtain the following result (see subsection~\ref{ss. examples}).

\begin{prop}
\label{p.example-Cantor}
There exists a transitive non minimal qpf circle homeomorphism whose unique
minimal set is a Cantor set, and whose intersection with each fibre
$\{\theta\} \times \kreis$ is uncountable.
\end{prop}

\bigskip

Apart from the Denjoy-like constructions, we collect some general properties of minimal sets of qpf circle homeomorphism. In particular, we prove the following uniqueness result (see subsection~\ref{ss.uniqueness-structure} for more results).

\begin{prop}
Suppose $f\in{\cal F}$ has no invariant strip. Then it has a unique minimal set. 
\end{prop}

Proposition~\ref{p.example-Cantor} shows that general qpf circle homeomorphisms may possess quite complicated minimal sets. In the particular case of quasiperiodic $\mbox{SL}(2,\RR)$-cocycles, which has received a lot of attention in the recent years (see, for example, \cite{avila/krikorian:2004} and references therein), we prove that such ``complicated" minimal sets cannot occur. More precisely, we obtain the following.

\begin{prop}[see propositions~\ref{p.twopoints},~\ref{p.sltr-onepoint},~\ref{p.bjerkloev/johnson} for more detailed statements]
\label{prop:cocycles}
Suppose $f$ is given by the projective action of a quasiperiodic $\mbox{SL}(2,\RR)$-cocycle.  Then any minimal set of $f$ 
\begin{enumerate}
\item is the whole torus,
\item or is a continuous ($p,q$)-invariant graph,
\item or intersects generic fibres in exactly one point,
\item or intersects generic fibres in exactly two points.
\end{enumerate}
If in addition $f$ is  $\rho$-unbounded, then any minimal set either is the whole torus, or intersects generic fibres in only one point.
\end{prop}

This proposition seems to improve some recent results of Bjerkl\"ov and Johnson (\cite{bjerkloev/johnson:2007}), by showing that one of the five possible cases of the classification obtained by these author never occurs. 

\bigskip

Finally, we want to mention another result which had originally been a
motivation for the presented work.
\begin{thm}[{\cite[theorem 4.4]{jaeger/keller:2006}}]\label{t.denjoy}
Suppose $f\in{\cal F}$ is ${\cal C}^2$ and has no invariant strips. Then $f$
is topologically transitive.
\end{thm}
Obviously, this raises again the question whether transitive but non-minimal
behaviour is possible in the absence of invariant strips. However, it must be
said that our results are not directly related to theorem~\ref{t.denjoy},
since we do not obtain examples with this type of regularity. As stated, we
are only able to choose the fibre maps ${\cal C}^1$, and by some slight
modifications one might push this to ${\cal C}^{1+\alpha}$ (see
section~\ref{Nu-Hoelder}). (Of course, the fact that a Denjoy-like construction
produces this type of regularity is by no means surprising.) Hence, the
question whether the assertion of theorem~\ref{t.denjoy} can be improved to
minimality is still open. In fact, this is not even known under much stronger
assumptions, for example if $f$ is real-analytic and $\omega$ is Diophantine,
or if $f$ is induced by the projective action of a quasiperiodic
$\textrm{SL}(2,\R)$-cocycle.

\section{Construction of the graph $\Gamma$}\label{section.curve}

\subsection{Graphs with flat intersections}
Let $R$ be a quasiperiodically forced circle homeomorphism over some circle
irrational rotation $\theta \mapsto \theta+\omega$.  
We will consider the graphs $\Gamma$ of continuous
 maps $\gamma: \kreis \to \kreis$. For
$I\subseteq \kreis$ we define $\Gamma_{|I} := \Gamma \cap ( I \times \kreis)$. 
The $C^0$-distance between continuous maps induces a distance $d$ between
graphs. By $p_1 : \thx \mapsto \theta$ we denote the canonical projection to the
first coordinate. 
\begin{definition}
Suppose $\gamma,\gamma': \kreis \to \kreis$ are two continuous maps with
graphs $\Gamma,\Gamma'$. 
  \label{def:flatintersections}
  \begin{list}{(\alph{enumi})}{\usecounter{enumi}}
  \item We say $\Gamma$ and $\Gamma'$ \emph{have flat intersections} if $p_1(\Gamma
  \cap \Gamma')$ consists of a finite union of disjoint intervals, none of
  which is reduced to a single point.
  \item We say $\Gamma$ and $\Gamma'$ \emph{cross over} some interval $I \subset
\kreis$ if there exists an interval $I' \subset I$ and an open interval $O \subsetneq
\kreis$ such that $\Gamma_{\mid I'}, \Gamma'_{\mid I'} \subset I' \times O$
and $\Gamma'_{\mid I'}$ meets both connected components of $(I' \times O)
\setminus \Gamma$.
  \end{list}
\end{definition}

\begin{prop}\label{prop.construction-curve}
Let $R$ be a quasiperiodically forced circle homeomorphism.  Assume that $R$
does not admit any continuous $(p,q)$-invariant graph.

Then there exists a continuous graph $\Gamma$ which has flat intersections with
all its iterates $R^{n}(\Gamma) \ (n\in\Z)$. 
\end{prop}
In the situation of theorem~\ref{t.mainthm}, the non-existence of continuous $(p,q)$-invariant graphs follows immediately from the minimality of $R$. However, the the fact that the above proposition holds under this weaker assumption will be useful in the later sections, and the proof is identical in both cases. 

\begin{adde}\label{addendum}
Assume that $R$ is topologically transitive.
Then the graph $\Gamma$ may be required to satisfy the following property:
\begin{description}
\item[(T)]  For all non-trivial intervals $I,J \subset \kreis$, there exists some $n \geq 0$
such that $\Gamma$ and $R^n(\Gamma)$ cross over $I \cap (J+n\omega)$.
\end{description}
\end{adde}

The graph $\Gamma$ required by proposition~\ref{prop.construction-curve}
will be the limit of a sequence $(\Gamma_{n})$, obtained by induction using proposition~\ref{prop.construction-curve-recurrence} below.
The section is organised as follows.
We first state proposition~\ref{prop.construction-curve-recurrence} and show that it entails 
proposition~\ref{prop.construction-curve}.
Sections~\ref{ss.perturbation},~\ref{subsection.minimal-real-analytic} and~\ref{subsection.general-case}
are mostly devoted to the proof of proposition~\ref{prop.construction-curve-recurrence}.
The addendum is proved at the end of section~\ref{subsection.general-case}.

We say a graph $\tilde{\Gamma}$ is an \emph{$\varepsilon$-modification} of a
graph $\Gamma$ over $I \ssq \kreis$ if $\tilde{\Gamma}_{|I^c} = \Gamma_{|I^c}$
and $d(\tilde{\Gamma},\Gamma) < \varepsilon$, where $d$ denotes the ${\cal
C}^0$-distance.
\begin{prop}\label{prop.construction-curve-recurrence}
Let $R$ be a quasiperiodically forced circle homeomorphism.  Assume that $R$
does not admit any continuous $(p,q)$-invariant graph.
Let $\Gamma$ be a continuous graph that has flat intersections with $R^k(\Gamma)$ for
$ 1 \leq k \leq n-1$.

Let $\varepsilon>0$.
Then there exists  a continuous graph $\Gammatil$ such that
\begin{enumerate}
\item $\Gammatil$ is an $\varepsilon$-modification of $\Gamma$ over a set of measure less than $\varepsilon$;
\item for $1 \leq k \leq n-1$,  the graph $\Gammatil$ has flat intersections with $R^k(\Gammatil)$, moreover 
the set $p_1(\Gammatil \cap R^k\Gammatil)$ contains $p_1(\Gamma \cap R^k\Gamma)$ and has the same number of connected components;
\item the graph $\Gammatil$ has flat intersections with $R^n(\Gammatil)$.
\end{enumerate}
\end{prop}

\begin{proof}[Proof of proposition~\ref{prop.construction-curve} using proposition~\ref{prop.construction-curve-recurrence}]
We begin the induction by choosing   $\Gamma_{1}$ to be any continuous graph. 
Let $n \geq 1$, and assume inductively that there exists a continuous graph $\Gamma_{n-1}$ which has flat intersection with  its iterates $R^k(\Gamma_{n-1})$  for all $k$ such that  $ 1 \leq k \leq n-1$. We apply proposition~\ref{prop.construction-curve-recurrence}, to get a continuous graph $\Gamma_{n} = \tilde \Gamma_{n-1}$ which has flat intersection with  its iterates $R^k(\Gamma_{n})$  for all $k$ such that  $ 1 \leq k \leq n$. Furthermore, we can demand that the graph  $\Gamma_{n}$ is a $\varepsilon_{n}$-modification of the graph $\Gamma_{n-1}$, with $\varepsilon_{n} \leq \frac{1}{2^n}$ (the choice of $\varepsilon_{n}$ will be made more precise below).

Thus we get a sequence $(\Gamma_{n})_{n \geq 1}$ of continuous graphs, which is a Cauchy sequence for the ${\cal C}^0$-distance. Let  $\Gamma$ be the limit map.
Let $k$ be a fixed positive integer. The sequence $(p_{1}(R^k(\Gamma_{n}) \cap \Gamma_{n}))_{n \geq 1}$ is an increasing sequence of subsets of $\kreis$, denote its limit by $I^k$,
$$
I^k = \mathrm{Clos}\left( \bigcup_{n \geq 1} p_{1}(R^k(\Gamma_{n}) \cap \Gamma_{n})\right).
$$ Note that according to property 2 of
 proposition~\ref{prop.construction-curve-recurrence}, every set in this
 sequence has the same number $a_{k}$ of connected components, so that $I^k$
 is again the disjoint union of at most $a_{k}$ non trivial compact intervals.

This set $I^k$ is included in $p_{1}(R^k(\Gamma) \cap \Gamma)$. To get the
reverse inclusion we have to make a more careful choice of the sequence
$(\varepsilon_{n})$.  For a fixed $n \geq k$, denote by ${\cal O}_{n,k}$ the
set of continuous graphs $\Delta$ such that $p_{1}(R^k(\Delta) \cap \Delta)$
is included in the $\frac{1}{n}$-neighbourhood of $p_{1}(R^k(\Gamma_{n}) \cap
\Gamma_{n})$ (which we denote by $V_{\frac{1}{n}} ( p_{1}(R^k(\Gamma_{n}) \cap
\Gamma_{n}))$). This set is open for the ${\cal C}^0$ distance.  Thus we may
have chosen the sequence $(\varepsilon_{n})_{n\in\N}$ so small that for every
$n \geq k$, $\Gamma \in {\cal O}_{n,k}$.
This entails, for every $k \leq n$,
$$ p_{1}(R^k(\Gamma_{n}) \cap \Gamma_{n}) \subset p_{1}(R^k(\Gamma) \cap
\Gamma) \subset V_{\frac{1}{n}} \left( p_{1}(R^k(\Gamma_{n}) \cap
\Gamma_{n})\right).
$$
Letting $n$ tends towards infinity (with fixed $k$) gives the required equality 
$I^k = p_{1}(R^k(\Gamma) \cap \Gamma)$.
Thus we get that $\Gamma$ has flat intersection with $R^k(\Gamma)$. 
\end{proof}
\subsection{Perturbation boxes}\label{ss.perturbation}
In this section, we  introduce the tools required by the proof of  proposition
\ref{prop.construction-curve-recurrence}.
We consider a continuous graph $\Gamma$ and some positive integer $n$
and make the following \emph{escaping hypothesis}.
\begin{center}
\emph{Every point has an iterate outside $\Gamma\cup R(\Gamma)\cup\dots\cup R^n(\Gamma)$.}
\end{center} 
In the case where $R$ is minimal this is obviously true, in
lemma~\ref{l.escaping} below we show that it also holds under the weaker
hypothesis of proposition~\ref{prop.construction-curve}~.

\subsubsection*{Returning dynamics on $\Gamma$}\label{sss.returning}
We introduce the \emph{first return map $T$ on $\Gamma$ in time less or equal to $n$}:
let $D$ be the set points $z$ in $\Gamma$ such that there exists some iterate
$R^q(z)$ in $\Gamma$ with $1 \leq q \leq n$; we then define the map $T$ on $D$
by $T(z) = R^q(z)$ where $q$ is the least such integer.

Due to the escaping hypothesis,  the orbit of any point of $\Gamma$ by $T$ is finite.
This allows the following definitions.
\medskip

\noindent \textbf{Notations and definitions.\quad} To any point $z\in \Gamma$
is associated a unique finite set, $N(z)=\{q_{-r}<\dots<
q_0=0<\dots<q_s\}\subset \ZZ$, called the \emph{itinerary of $z$}, such that
\begin{itemize}
\item  for any integer $k$ such that $q_{-r}-n\leq k \leq q_s+n$, the point $R^k(z)$
belongs to $\Gamma$ if and only if $k=q_i$ for some $-r\leq i\leq s$,
\item  $|q_i-q_{i+1}|\leq n$ for each $-r\leq i<s$.
\end{itemize}
We thus have $T^i(z)=R^{q_{i}}(z)$, and the sequence $(R^{q_{i}}(z))_{i=-r,
\dots ,s}$ is the \emph{$T$-orbit} of $z$.  We will denote by $\ell(z) = s +
r$ the length of this orbit.  When we wish to emphasize the dependence on the
point $z$, we will write $r(z)$, $s(z)$, and so on. Let us note that the
$T$-orbit of $z$ is reduced to $(z)$ if and only if all the iterates $f^k(z)$
are outside $\Gamma$ for $0 < \mid k \mid \leq n$.

More generally, for any sufficiently small interval $I \subset \kreis$, there
exists a finite set $ \{ q_{i}, i=-r , \dots, s\}, $ again called the
\emph{itinerary of $I$}, such that $0<q_{i+1} - q_{i} \leq n$ and for all $k$
such that $q_{-r}-n \leq k \leq q_{s}+n$ we have
$$
R^{k}(\Gamma_{\mid I}) \cap \Gamma \neq \emptyset \Leftrightarrow k \in 
 \{ q_{i}, i=-r , \dots, s\}.
$$ This follows from compactness together with the openness with respect to
$z$ of the property $R^k(z) \notin \Gamma$ for a given $k$.

\subsubsection*{Definition and existence of perturbation boxes}

\begin{definition}\label{def.boxes}
A rectangle $B= I \times J$ is a \emph{perturbation box} if
\begin{enumerate}
\item $I$ has a finite itinerary  $\{ q_{i}, i=-r , \dots s\}$
and has pairwise disjoint iterates
$I+k\omega$ with $q_{-r}-n \leq k \leq q_{s}+n$;
\item one of the two endpoints of $\Gamma_{\mid I}$  has the same itinerary  as  $I$;
\item for all $k$ such that $q_{-r}-n \leq k \leq q_{s}+n$, the graph
$\left(R^{-k}(\Gamma)\right)_{|I}$
\begin{itemize}
\item is contained in $B$ if $k = q_{i}$ for some $-r \leq i \leq s$,
\item is disjoint from $B$ otherwise.
\end{itemize}
\end{enumerate}
\end{definition}
We define the itinerary of the box to be the itinerary of $I$.  Next we prove
that every points of $\Gamma$ belong to a perturbation box.

\begin{lemm}\label{lemma.exists-perturbation-boxes}
If $z=(\theta,x) \in \Gamma$, then there exists $\delta ,\eta$ arbitrarily
small such that $B = [\theta,\theta+\delta ] \times [x-\eta,x+\eta]$ is a
perturbation box whose itinerary coincides with the itinerary of $z$.
\end{lemm}

\begin{proof}[Proof of the lemma]
For $\delta$ and $\eta$ small enough, the rectangle $B$ is disjoint from the
graphs $R^{-k}(\Gamma)$ with $q_{-r}-n\leq k\leq q_s + n $ such that $k$ does
not belong to the itinerary of $z$. For $k=q_{i}$, the graph $R^{-k}(\Gamma)$
contains the point $z$; hence, if the interval $\delta $ is chosen after
$\eta$ and small enough, then
$\left(R^{-k}(\Gamma)\right)_{|[\theta,\theta+\delta ]}$ is contained in $B$.
\end{proof}

\subsubsection*{Perturbation lemma}
If one assumes that the intersections $\Gamma\cap R^k(\Gamma)$ are controlled
for any $k$ up to $n-1$, the perturbation boxes can be used to build a
perturbed graph $\Gamma'$ whose intersections $\Gamma'\cap R^k( \Gamma')$ are
flat for $k$ up to $n$.  These perturbations are given by the following lemma.
\begin{lemm}[Perturbation lemma]\label{l.perturbation}
Let $B=I \times J$ be a perturbation box. Denote by $\{ q_{i}, i=-r , \dots,
s\}$ the itinerary of $I$.  Then there exists a perturbation $\Gamma'$ of
$\Gamma$ such that
\begin{enumerate}
\item the perturbation is supported in 
$\displaystyle \bigcup_{i=-r}^s R^{q_{i}}(B)$, and 
in particular the graphs $\left(R^{-q_{i}}(\Gamma')\right)_{|I}$ are still included in $B$.
\item Let  $1\leq k\leq n$ and define the sets  $X_{k}= p_1(\Gamma \cap R^k\Gamma)$
and $X'_{k}= p_1(\Gamma' \cap R^k\Gamma')$. Then
$X'_{k}$ is the union of $X_{k}$ and a finite number of intervals meeting $X_{k}$.
More precisely,
for every $i$ such that  $X_{k}$ meets $I+q_{i}\omega$,
there exists a non-trivial  interval $J_{i,k}$ satisfying
$X_{k} \cap (I+q_{i}\omega) \subset J_{i,k}  \subset I+q_{i}\omega$,
and we have the equality
$$
X'_{k}= X_{k} \cup  \bigcup_{X_{k} \cap( I+q_{i}\omega )\neq \emptyset} J_{i,k}.
$$ 
\end{enumerate}
Furthermore, if $R$ is real-analytic and $\Gamma$ is piecewise real-analytic,
then $\Gamma'$ can be chosen piecewise real-analytic.
\end{lemm}

\begin{proof}[Proof of the perturbation lemma \ref{l.perturbation}]
We want to construct the new graph $\Gamma'$ by modifying the graph $\Gamma$ above each interval $I+q_{i}\omega$.
Since all the intervals $I+q_{i}\omega$ are pairwise disjoint (item 1 of the definition of perturbation box), this amounts to modifying each  graph $R^{-q_{i}}(\Gamma)$ above $I$.

We denote by $z$ the endpoint of $\Gamma_{\mid I}$ that has the same itinerary
as $I$ (item 2 of the definition of perturbation box).  To fix ideas, we
assume that $z$ is the left endpoint (the proof is entirely similar if $z$ is
the right endpoint).  The definition of the itinerary entails that $z$ also
belongs to $R^{-q_{i}}(\Gamma)$ for all $i$.  Denote $I =
[\theta,\theta+\delta ]$. Choose some $ \lambda \in (0,\delta)$ with the
property that for any given $i,j \in \{-r,\ldots,s\}$ the graphs
$R^{-q_{i}}(\Gamma)_{\mid[\theta+\lambda,\theta+\delta]}$ and
$R^{-q_{j}}(\Gamma)_{\mid[\theta+\lambda,\theta+\delta]}$ either coincide at
$\theta+\delta$ or are disjoint.  Above $I$, we replace each graph
$R^{-q_{i}}(\Gamma)$ by the union of two straight segments:
\begin{itemize}
\item above $[\theta,\theta+\lambda]$, the new graph is a horizontal segment (starting at $z$) ;
\item above $[\theta+\lambda, \theta+\delta ]$, the new segment is forced by continuity
(it connects the right endpoint of the first segment
to the  point of $R^{-q_{i}}(\Gamma)$ above $\theta+\delta$).
\end{itemize}

Let us check the new graph $\Gamma'$ has the announced properties. The first
one concerning the support is a consequence of item 3 of the definition of a
perturbation box.  Let us turn to the second one.  For any $i,j $ between $-r$
and $s$, let us define the set
$$I_{i,j} = p_{1}(R^{-q_{i}}(\Gamma') \cap R^{-q_{j}}(\Gamma'))\cap I.$$
This set is either equal to $[\theta,\theta+\lambda]$ or to $I$.
In case $p_{1}(R^{-q_{i}}(\Gamma) \cap R^{-q_{j}}(\Gamma))\cap I = I$
then we still have $I_{i,j}=I$. 
By the choice of $\lambda$, in the opposite cases,
$[\theta,\theta+\lambda]$ contains 
$p_{1}(R^{-q_{i}}(\Gamma) \cap R^{-q_{j}}(\Gamma))\cap I$, and then so does $I_{i,j}$.
Note that in any case we have the following property:
\begin{equation}
\tag{*}\begin{array}{l}\textit{For all } i,j \in \{-r , \dots , s\}, \textit{
the interval } I_{i,j} \textit{ is a non-trivial interval} \\
\textit{containing } p_{1}(R^{-q_{i}}(\Gamma) \cap R^{-q_{j}}(\Gamma))\cap I \
. \end{array}
\end{equation}

\begin{claim}
For every $k=1,\dots,n$ and every $i = -r, \dots, s$,
$$
X_{k} \cap( I+q_{i}\omega )\neq \emptyset 
\ \Leftrightarrow \  \exists j, \ k=q_{i} - q_{j}.
$$
\end{claim}
When these equivalent properties hold we define $J_{i,k} = I_{i,j}
+q_{i}\omega$. Note that according to property (*), $J_{i,k}$ contains $X_{k}
\cap (I+q_{i}\omega)$. Obviously $J_{i,k}$ is contained in $I+q_{i}\omega$.

\begin{proof}[Proof of the claim]
$$
\begin{array}{rcll}
\exists j, \ k=q_{i} - q_{j}  & \Leftrightarrow   &  R^{k-q_{i}}(\Gamma) \cap B \neq \emptyset   &  \mbox{\footnotesize (item 3 of definition~\ref{def.boxes})} \\
 & \Leftrightarrow   &   z \in R^{k-q_{i}}(\Gamma)   &  \mbox{\footnotesize (choice of $z$)}  \\
  & \Leftrightarrow   &    R^{k-q_{i}}(\Gamma) \cap \left(R^{-q_{i}}(\Gamma)\right)_{\mid I} \neq\emptyset   \\
  & \Leftrightarrow   &  X_{k} \cap( I+q_{i}\omega )\neq \emptyset    &  \mbox{\footnotesize (apply $R^{q_{i}}$)} .
    \end{array}
$$

\end{proof}

Now for getting property 2 it only remains to checking 
 the following equality:
$$
X'_{k} = X_{k} \cup  \bigcup_{q_{i} - q_{j} = k}  \left(  I_{i,j} +q_{i}\omega \right).
$$
Let us  define the set $J = \bigcup_{i=-r}^{s} I+q_{i}\omega$.
In order to  check the above equality 
 we partition $\kreis$ into  the four sets
 $$
 J \cap ( J+k\omega ) , \quad J \setminus ( J+k\omega ) , \quad ( J+k\omega )  \setminus J, \quad \kreis \setminus \left( J \cup ( J+k\omega )  \right).
 $$
Let us examine the first set $J \cap ( J+k\omega ) $.
According to item 1 of the definition of the perturbation boxes, we have 
$$J \cap ( J+k\omega ) = \bigcup_{q_{i} - q_{j} = k} \left( I +q_{i}\omega
\right).$$ Let $i,j$ be such that $q_{i}-q_{j}=k$.  Restricted to
$I+q_{i}\omega$, we have $X'_{k} = I_{i,j} +q_{i}\omega$, and, according to
property (*), this set contains the restriction of $X_{k}$: in other words,
$$ X'_{k} \cap \left(I+q_{i}\omega\right) = \left[ X_{k} \cup ( I_{i,j}
+q_{i}\omega ) \right] \cap \left(I+q_{i}\omega\right).
$$ It remains now to check that outside $J \cap ( J+k\omega ) $, the sets
$X'_{k}$ and $X_{k}$ coincides.  The second set $J \setminus ( J+k\omega ) $
is the union of the intervals $I+q_{i}\omega$ for those $i$ such that $q_{j}
\neq q_{i}-k$ for every $j$. Choose such an $i$.  According to item 3 of the
definition of perturbation boxes,
$$ \Gamma_{\mid I+q_{i}\omega} \subset R^{q_{i}}(B) \mbox{ and }
\left(R^k(\Gamma)\right)_{\mid I+q_{i}\omega} \cap R^{q_{i}}(B) =\emptyset,
$$ and in particular $X_{k} \cap (I+q_{i}\omega ) $ is empty.  According to
item 1 of the lemma, the same relations hold when $\Gamma$ is replaced by
$\Gamma'$, thus $X'_{k} \cap (I+q_{i}\omega ) $ is also empty.  Thus $X_{k}$
and $X'_{k}$ coincide in restriction to $J \setminus ( J+k\omega ) $.

Restricting to the third set $( J+k\omega ) \setminus J$, we prove
symmetrically that both sets $X_{k}$ and $X'_{k}$ are empty, and so they also
coincide.

On the last set $\kreis \setminus \left( J \cup ( J+k\omega ) \right)$,
according to item 1 of the lemma, we have $ \Gamma' = \Gamma $ and $
R^k(\Gamma') = R^k(\Gamma) $ and thus $ X'_{k} = X_{k} $.

\end{proof}

\begin{bem}\label{r.perturbation}
In the proof of the perturbation lemma~\ref{l.perturbation}, we chose to
replace the graph $\Gamma$ by the simplest possible curve, \emph{i. e.} the
concatenation of two segments.  But of course we could have used more
complicated curves, for example the concatenation of a finite number of
segments. Thus, if we are given some point $z$ within the interior of the
perturbation box $B$, this modification allows us to force the perturbed
curve $\Gamma'$ to contain the point $z$. The same holds for any finite number
of points in $\inte(B)$ (obviously having distinct first coordinate).
\end{bem}

\subsection{Construction of the graph $\Gamma$: the real-analytic minimal case}\label{subsection.minimal-real-analytic}
The construction of the graph $\Gamma$ is easier if $R$ is a minimal rotation,
or more generally if $R$ is minimal and real-analytic.  To explain this easy
case, we state a simpler version of proposition
\ref{prop.construction-curve-recurrence}.

\begin{prop} \label{p.analytic-perturbation}
Suppose that the assumptions of proposition~\ref{prop.construction-curve-recurrence} hold and in addition:
\begin{itemize}
\item  $R$ is real-analytic  and minimal,
\item  $\Gamma$ is piecewise analytic.
\end{itemize}
Then there exists a curve $\Gammatil$ which satisfies the assertions proposition~\ref{prop.construction-curve-recurrence} 
and is piecewise analytic. 
\end{prop}
\medskip

Replacing proposition~\ref{prop.construction-curve-recurrence} by the preceding one (which is much easier to show), the proof of proposition~\ref{prop.construction-curve}
given at the beginning of section~\ref{section.curve} can be easily adapted to prove the statement in the case when $R$ is minimal 
and real-analytic.

\begin{proof}[Proof of proposition~\ref{p.analytic-perturbation}]
Using that $R$ is real-analytic and $\Gamma$ is piecewise real-analytic,
we know that the intersection $\Gamma \cap R^n\Gamma$ has finitely
many connected components: these are isolated points and non-trivial curves.
We will explain how to build a modification $\Gamma'$ that is piecewise real-analytic,
satisfies items 1 and 2 of the proposition \ref{prop.construction-curve-recurrence}
and moreover
\begin{itemize}\it
\item[3-bis.] the number of connected components of
$\Gamma' \cap R^n \Gamma'$ that are reduced to a point
is strictly less than
the corresponding number for $\Gamma \cap R^n\Gamma$.
\end{itemize}
By repeating  this construction finitely many times, one obtains a modification $\Gammatil$
of $\Gamma$ that now satisfies the item 3 of the proposition.

Since $R$ is supposed to be minimal,
the escaping property is satisfied,
and we can apply  section~\ref{ss.perturbation}.
Let us consider an isolated point $z\in \Gamma\cap R^n(\Gamma)$.
By lemma~\ref{lemma.exists-perturbation-boxes} there exists an arbitrarily small perturbation box $B$
containing $z$ in its boundary and having the same itinerary as $z$.
We now apply the perturbation  lemma~\ref{l.perturbation}.
Since $B$ is arbitrarily small, the perturbation $\Gamma'$ is small: item 1 of proposition~\ref{prop.construction-curve-recurrence}
is satisfied. One can also assume that the width $p_{1}(B)$ is smaller
than half of the distance between any two connected components of $\Gamma\cap R^k(\Gamma)$
for any $1\leq k\leq n$.
Item 2 of lemma~\ref{l.perturbation} implies that $\Gamma'\cap R^k(\Gamma')$
contains $\Gamma\cap R^k(\Gamma)$ and has the same number
of connected components. When $k<n$, one gets item 2 of the proposition.
One also obtains  item 3-bis by noting that the component $\{z\}$ of $\Gamma\cap R^n(\Gamma)$
has been replaced by a non-trivial interval. This completes the construction of $\Gamma$ in the real-analytic case.
\end{proof}

\subsection{Construction of the graph $\Gamma$: the general case}
\label{subsection.general-case}
In this section we prove proposition~\ref{prop.construction-curve-recurrence} and the addendum \ref{addendum} to proposition~\ref{prop.construction-curve}~.
\subsubsection*{Escaping hypothesis}
We start by checking the escaping hypothesis under which points have finite itineraries (see section~\ref{ss.perturbation}).
\begin{lemm} \label
{l.escaping}
Assume $R$ does not admits any continuous $(p,q)$-invariant graph.
Then for every continuous graph $\Gamma$ and every $n>0$,
the escaping hypothesis is satisfied.
\end{lemm}

\begin{proof}
If the escaping hypothesis is not satisfied, then 
 there exists a  invariant compact set $K$ included in the union $\hat K$ of a finite number of iterates of some graph $\Gamma$.
 Let us choose a minimal such $K$. Then the compact set $K$ is a continuous $(p,q)$ invariant graph contradicting the assumption.
 Indeed,
since $K$ is compact and invariant, $p_{1}(K)$ is the whole circle ; thus there is some $k$ such that the projection of $R^k(\Gamma) \cap K$ has non-empty interior
(Baire theorem), so that $K$ contains some graph $\alpha$ over some interval.
By taking a smaller open interval, we find a graph over an open interval which is an open set of $K$. By minimality, $K$ is a one-dimensional topological manifold,
thus a union of simple curves. The same argument shows that $p_{1}$ is a local homeomorphism on $K$,
and thus a covering map (by compactness). 
 \end{proof}
\subsubsection*{Proof of proposition~\ref{prop.construction-curve-recurrence}}

The proposition is obtained by applying inductively the following lemma.
This lemma roughly says that if $R^n(\Gamma)$ has flat intersections with $\Gamma$
 outside some closed subset $F$, then we can construct $\Gammatil$ such that  
 $R^n(\Gammatil)$ has flat intersections with $\Gammatil$
 outside a closed subset $F'$, where $F'$ is  substantially smaller than $F$.
Remember that the map $T$ and the function $\ell$ have been defined on section~\ref{sss.returning}; we will denote by $T'$ and $\ell'$ 
the corresponding objects with respect to the graph $\Gamma'$.

\begin{lemm}\label{lemma.zigotto}
Let $\Gamma$ be a continuous graph that has flat intersections with $R^k(\Gamma)$ for
$ 1 \leq k \leq n-1$.
Let $F$ be a non-empty closed  set  which is a union of $T$-orbits.
Suppose that 
$\left( \Gamma \cap R^{n}(\Gamma)\right)  \setminus F$
has a finite number of connected components, none of which is a single point.

Let $\varepsilon_{0} >0$. Then there exists an $\varepsilon_{0}$-perturbation 
$\Gamma'$ of $\Gamma$ supported on an arbitrarily small neighbourhood of
$F$, and there exists a  non-empty closed  set  $F' \subset F$  which is a union of $T$-orbits
such that
\begin{enumerate}
\item $\left( \Gamma' \cap R^{n}(\Gamma')\right)  \setminus F'$
has a finite number of connected components, none of which is a single point.

\item For $1 \leq k \leq n-1$,  the graph $\Gamma'$ has flat intersections with $R^k(\Gamma')$, moreover 
the set $p_1(\Gamma' \cap R^k\Gamma')$ contains $p_1(\Gamma \cap R^k\Gamma)$ and has the same number of connected components.

\item Either the set $F'$ is empty, or  the supremum of the function $\ell'$ on $F'$ is strictly less than the supremum of $\ell$ on $F$.
\end{enumerate}
\end{lemm}

\begin{proof}[Proof of the lemma]
We assume the hypotheses of the lemma.
If $\ell(z)=0$ for every $z \in F$, then $R^n(\Gamma) \cap F = \emptyset$,
and the graph $\Gamma'=\Gamma$ together with its closed subset $F'=\emptyset$
satisfies the conclusion of the lemma.
>From now on we assume that $\sup_{F}\ell >0$.

Let 
$$
M=\{ z \in F, \ \ \ \ell(z) = \sup_{F}\ell \}.
$$
The set $M$ is obviously a non-empty union of  $T$-orbits.
\begin{claim}
The set $M$ is closed.
\end{claim}
To prove the claim we consider the map $z \mapsto N(z)$
which associates to each point of $\Gamma$ its itinerary (see section~\ref{sss.returning}).
Let $z \in \Gamma$; by definition of the itinerary, 
the points $f^k(z)$ are outside $\Gamma$ for every $k \in \{q_{-r(z)}-n, \dots , q_{s(z)}+n\} \setminus N(z)$.
We note that  these conditions are open, 
so  we have $N(z') \subset N(z)$ for all $z'$ in a neighbourhood of $z$
(in other words,   the map $z \to N(z)$ is semi-continuous).
The claim follows.
 
 We consider the finite partition $\cal P$ of $M$ induced by $N$:
 two points $y$ and $z$ are in the same element $P$ of $\cal P$ 
 if and only if $N(y) = N(z)$. Since the functions $N$ and $\ell$ are constant on $P$ we use the notation $N(P)$ and $\ell(P)$.
 The semi-continuity of $N$ and the maximality of $\ell$ on $M$ also entail
  that the elements of this partition are closed sets.
  Let ${\cal P}_{0}$ be the family of $P \in \cal P$ whose points $z$ satisfy $r(z)=0$.
Then the partition $\cal P$ can be written as
\begin{equation*}
{\cal P} = \{ T^k (P), P\in {\cal P}_{0} \mbox{ and }  0 \leq k \leq \ell(P)    \}
 = \{ R^{q} (P), P \in {\cal P}_{0} \mbox{ and }  q\in N(P)    \}.
 \end{equation*}

\begin{claim}
Let $P \in {\cal P}_{0}$. 
There exists a finite collection ${\cal B}({P})$ of arbitrarily small perturbation boxes  such that
\begin{enumerate}
\item each box $B$ has the same itinerary $N(B)$ as the points of $P$,
\item 
the projections $p_{1}(B)$ of the boxes have pairwise disjoint interior,
\item $p_{1}(P)$ is contained in the interior of the union of all the $p_{1}(B)$.
\end{enumerate}
\end{claim}
\begin{proof}[Proof of the claim]
Let $z$ be any point of $P$.
By lemma~\ref{lemma.exists-perturbation-boxes},
we can find two arbitrarily small perturbation boxes $B^-(z) = I^-(z) \times J^-(z), B^+(z) = I^+(z) \times J^+(z)$ whose itineraries coincide with the itinerary of $z$, and such that $p_{1}(z)$ is the right end-point of $I^-(z)$ and the left end-point of $I^+(z)$.
Since $p(P)$ is compact it is covered by the interior of finite number of intervals $I^-(z)\cup I^+(z)$.
Thus we find a finite collection ${\cal B}^*({P})$ of perturbation boxes having properties 1 and 3 of the claim but maybe not property 2.
Then property 2 will be achieved by replacing some of the intervals $I^{\pm}(z)$ by smaller subintervals. For this we first make the following remark. 
Let $z=(\theta_{0},x) \in P$, and $\theta_{1} \in \inte (I^-(z)) = (\theta_{0}-\delta,\theta_{0})$. Then 
\begin{itemize}
\item  $[\theta_{1},\theta_{0}] \times J^-(z)$ is a perturbation box,
\item if $\theta_{1} \in p_{1}(P)$ then $[\theta_{0}-\delta ,\theta_{1}] \times J^-(z)$ is a perturbation box,
\end{itemize}
and in both cases the new box has the same itinerary as the old one.
Now consider any couple of perturbation boxes $B_{1},B_{2} \in {\cal B}^*({P})$ whose interiors are not disjoint.
If the $p_{1}$-projection of one box contains the other one then we can just eliminate the smallest one. In the opposite case, the above remark allows us to replace one of the two boxes, say $B_{1}$, by a smaller perturbation box $B'_{1}$ whose projection by $p_{1}$ is disjoint from $B_{2}$
and such that the $p_{1}$-projection of $B'_{1} \cup B_{2}$ equals the $p_{1}$-projection of $B_{1} \cup B_{2}$.  
Thus, by considering one by one all the couples of boxes in ${\cal B}^*({P})$,
 we can construct a new collection ${\cal B}({P})$ having the wanted properties.
\end{proof}

Let ${\cal B} = \cup {\cal B}({P})$ be the family of all the constructed boxes.
Since the partition $\cal P$ of $M$
consists of disjoint closed sets, each $P \in {\cal P}_{0}$ is contained in an open set  $U(P)$ such that the collection 
$$
\{ R^{q} (U(P)), P \in {\cal P}_{0} \mbox{ and }  q\in N(P)    \} 
$$
still consists of pairwise disjoint sets.
Thus items 2 and 3 of the claim allows us to choose the families ${\cal B}({P})$ of perturbation boxes  such that the elements of the following  family have disjoint interior:
$$
 \{ p_{1}(B)+q\omega, B \in {\cal B}    \mbox{ and }  q \in N(B)   \}.
$$
Thus we can apply the perturbation lemma~\ref{l.perturbation} independently on each box of the family $\cal B$,
and denote by $\Gamma'$ the resulting graph.
The graph $\Gamma$ hence has been modified only in the domain
$$
Z =  \bigcup_{\stackrel{B \in {\cal B}}{q \in N(B)}}  p_{1}(R^q(B))\times \kreis.
$$
We define  the set $F'$ by
$F' = F \setminus \inte (Z)$ and we now check the properties.

The set $F'$ is clearly a closed subset of $F$.
For any $B \in {\cal B}$, and any $z \in B$, the itinerary of $z$ is a subset of the itinerary of $B$.
Thus the full $T$-orbit of $z$ is included in the union $\bigcup_{q \in N(B)}  R^q(B)$.
Also note that the maps $T$ and $T'$ coincide outside $\inte(Z)$, and thus on $F'$. Consequently, 
$F'$ is a union of non-trivial  $T'$-orbits.

Item 2 is obvious (\emph{cf} property 2 of the perturbation lemma~\ref{l.perturbation}).

Let us check item 1.
We first analyse the set $(\Gamma' \cap R^n( \Gamma') ) \setminus F' $
in restriction to the complementary set of $Z$. We have
$$
\left( (\Gamma' \cap R^n(\Gamma') ) \setminus F'  \right )  \setminus Z=
\left( ( \Gamma \cap R^n( \Gamma) ) \setminus F  \right )  \setminus Z.
$$
The right-hand set is the restriction of $( \Gamma \cap R^n( \Gamma) ) \setminus F$
(the union of finitely many non-trivial closed graphs) above the union of finitely many open intervals;
thus it is the union of a finite number of non trivial intervals.
Then we analyse the set $(\Gamma' \cap R^n( \Gamma') ) \setminus F' $ in restriction to $Z$. According to lemma~\ref{l.perturbation}, 
$\Gamma' \cap R^n(\Gamma') \cap Z$ has a finite number of connected components, all of them non-trivial. Furthermore, by definition of $F'$, the set $F' \cap Z$ is included in the boundary of $Z$ and is finite.
Thus $((\Gamma' \cap R^n(\Gamma') ) \setminus F' ) \cap Z$ is again the union of a finite number of non trivial intervals.
Putting everything together, we get item 1.

Finally for item 3 we note that $M \subset \inte(Z)$ and thus $F'$ does not meet $M$.
Since the maps $T$ and $T'$ coincide on $F'$, and according to the definition of the set $M$, we have 
$$
 \sup_{F'}\ell' =  \sup_{F'} \ell  <  \sup_{F}\ell \ .
$$
This completes the proof.
\end{proof}

\begin{proof}[Proof of  proposition~\ref{prop.construction-curve-recurrence} from lemma~\ref{lemma.zigotto}]
Let $\Gamma$ and $\varepsilon$ be as in the hypotheses of the proposition.
We fix once and for all the value $\varepsilon_{0} = \frac{\varepsilon}{\sup_{\Gamma}\ell +1}$.
We set $\Gamma_{0}= \Gamma$ and $F_{0}= \Gamma$.
Then $\Gamma_{0}$ and $F_{0}$ satisfy the hypotheses of the lemma.
Applying the lemma provides a new graph $\Gamma_{1} = \Gamma'$ with a closed subset $F_{1} = F'$.
If $F'$ is empty, then we define $\Gammatil=\Gamma_{1}$ and note that $\Gammatil$
satisfies the conclusion of proposition~\ref{prop.construction-curve-recurrence}.

In the opposite case, $\Gamma_{1}$ and $F_{1}$ again satisfy the hypotheses of the lemma.
Applying the lemma recursively provides  sequences $(\Gamma_{p})$ and $(F_{p})$.
Since the inequality   $\sup_{F_{p+1}}\ell_{p+1} < \sup_{F_{p}}\ell_{p}$ holds, there exists some $p$ with $F_{p}= \emptyset$,
and then we can define the graph $\Gammatil = \Gamma_{p}$.
Note that $p\varepsilon_{0} < \varepsilon$, so that points 1 and 2 of the proposition concerning the size of the perturbation hold.
This completes the proof of the proposition.
\end{proof}

\subsubsection*{Proof of  addendum~\ref{addendum}}\label{subsubsection.proof-addendum}

Addendum~\ref{addendum} requires the following additional property for the graph $\Gamma$:
for any small pieces $\Gamma_{\mid I}$, $\Gamma_{\mid J}$, there exists some positive iterate of the first one that crosses the second one. In order to get this additional property, we will refine the construction of the sequence $(\Gamma_{n})$, by inserting between two successive steps of the construction  another small modification. 
The modification between steps $n$ and $n+1$ will achieve the wanted property concerning two specific intervals $I_{n},J_{n}$, while keeping the previously obtained properties of the graph $\Gamma_{n}$. We will get the modification by applying the following proposition.

\begin{prop}\label{prop.transitivity}
Assume that $R$ is topologically transitive.
Let $\Gamma$ be a continuous graph that has flat intersections with $R^k(\Gamma)$ for
$ 1 \leq k \leq n-1$.

Let $\varepsilon>0$, and $I,J \subset \kreis$ be  two non-trivial intervals.
Then there exists  a Graph $\Gammatil$ such that
\begin{enumerate}
\item $\Gammatil$ is an $\varepsilon$ modification of $\Gamma$ over a set of measure less than $\varepsilon$;
\item there exists some integer $m >0$ such that $\Gamma$ and $R^m(\Gamma)$ crosses over $I\cap (J + m \omega)$;
\item for $1 \leq k \leq n-1$,  the graph $\Gammatil$ has flat intersections with $R^k(\Gammatil)$, moreover 
the set $p_1(\Gammatil \cap R^k\Gammatil)$ contains $p_1(\Gamma \cap R^k\Gamma)$ and has the same number of connected components.
\end{enumerate}
\end{prop}
\begin{proof}[Proof of proposition~\ref{prop.transitivity}]
We first note that, up to replacing $I$ and $J$  by subintervals, we can assume the following additional properties:
\begin{itemize}
\item $I$ and $J$ have finite itineraries $\cal I$ and $\cal J$;
\item there exists vertical intervals $I',J' \subset \kreis$ such that the rectangles $B_{I}=I \times I'$ and  $B_{J}=J \times J'$ are perturbation boxes;
\item any two intervals $I + k \omega, k \in {\cal I}$ and  $J+\ell \omega, \ell \in {\cal J}$ are disjoint.
\end{itemize}

Since $R$ is topologically transitive, there exists a point $z \in \inte(B)$ with a dense forward orbit.
Now we apply the perturbation lemma~\ref{l.perturbation} with the perturbation box $B_I$
to construct a first  $\varepsilon$-modification  $\bar \Gamma$ of $\Gamma$ that contains $z$ (this is possible thanks to remark~\ref{r.perturbation}). Note that the perturbation is supported on the iterates of $B_{I}$ corresponding to the itinerary of $B_{I}$, and thus $\bar \Gamma \cap B_{J} = \Gamma \cap B_{J}$.

By hypothesis on $z$ there is an iterate $R^m(z)$ with positive $m$ belonging
to $\inte(B_{J})$.  Let $z_{1},z_{2}$ be two points of $\inte(B_{J} \cap (
(I + m \omega) \times \kreis ))$ that are separated in $B_{J} \cap ( (I + m \omega) \times
\kreis )$ by $R^m(\bar \Gamma)$. Now we perform a second modification of
$\Gamma$, this time using the box $B_{J}$, to construct a new graph $\tilde
\Gamma$ that contains both points $z_{1}, z_{2}$. Since $\tilde \Gamma \cap
B_{I} = \bar \Gamma \cap B_{I}$, it follows that $R^m(\tilde\Gamma)_{(I + m \omega)} =
R^m(\bar \Gamma)_{(I + m \omega)}$: thus $z_{1}$ and $z_{2}$ are still separated in
$B_{J} \cap ( (I + m \omega) \times \kreis )$ by $R^m(\tilde\Gamma)$, that is,
$\tilde \Gamma$ and $R^m(\tilde \Gamma)$ cross over $(I + m \omega) \cap J$.
\end{proof}

\begin{proof}[Proof of addendum~\ref{addendum} using proposition~\ref{prop.transitivity}]
We explain how to modify the inductive proof of proposition~\ref{prop.construction-curve-recurrence} given at the beginning of section~\ref{section.curve} so that the addendum is satisfied.

Let $(I_{i})$ be a countable basis for the topology of $\kreis$ which consists of  intervals. Let
$(i_{n},j_{n})_{n \geq 1}$ be an enumeration of $\N^2$, so that $(I_{i_{n}} \times I_{j_{n}})_{n \geq 1}$ is a basis for $\kreis \times \kreis$.

Let $n \geq 1$, and assume inductively that there exists a continuous graph
$\Gamma_{n-1}$ which has flat intersections with its iterates
$R(\Gamma_{n-1}), \dots , R^{n-1}(\Gamma_{n-1})$. We first mimic the
construction of proposition~\ref{prop.construction-curve-recurrence} to get a
continuous graph which we denote by $\Gamma'_{n}$; in particular,
$\Gamma'_{n}$ is an $\varepsilon_{n}$-modification of $\Gamma_{n-1}$ and has
flat intersection with its $n$ first iterates.

Thus we can apply proposition~\ref{prop.transitivity} to this graph
$\Gamma'_{n}$, with intervals $I=I_{i_{n}}$ and $J=I_{j_{n}}$. We get again a
continuous graph $\Gamma_{n}$ wich is an $\varepsilon'_{n}$-modification of
$\Gamma'_{n}$ (property 1).  This new graph still has flat intersections with
its $n$ first iterates (property 3). Furthermore, it satisfies the following
additional property (property 2): there exists some integer $m >0$ such that
$\Gamma_{n}$ and $R^m(\Gamma_{n})$ crosses over the interval $I_{i_{n}} \cap
(I_{j_{n}} + m \omega )$.

We note that this last property is an open property among continuous graphs
 for the ${\cal C}^0$ distance. Thus, by choosing the sequences
 $(\varepsilon_{n})$ and $(\varepsilon'_{n})$ to decrease sufficiently fast,
 we can ensure that the limit graph $\Gamma$ also satisfies this property :
 for each integer $n\geq 1$ there exists an integer $m >0$ such that $\Gamma$
 and $R^m(\Gamma)$ crosses over the interval $I_{i_{n}} \cap (I_{j_{n}} + m
 \omega )$.  In particular, $\Gamma$ meets the property required by the
 addendum. Furthermore, the argument of
 proposition~\ref{prop.construction-curve-recurrence} showing the flat
 intersection of $\Gamma$ with $R^n(\Gamma)$ for all $n$ is still valid. This
 completes the proof of the addendum.
\end{proof}


\section{Blowing up the orbit of $\Gamma$} \label{Blow-up-curve}
In this section, we consider a quasiperiodically forced circle homeomorphism
$R$ and a continuous graph $\Gamma$ which has flat intersections with all its
iterates $R^{n}(\Gamma) \ (n\in\Z)$. We denote the fibres of $p_1$ by
$\TT^1_{\theta}=\{\theta\} \times \TT^1$.

Let us first recall that every probability measure $\mu$ on $\TT^2$ can be
disintegrated with respect to fibres of the projection $p_1:(\theta,x)\mapsto
\theta$. More precisely, one can find a family $(\mu_\theta)_{\theta\in\TT^1}$
of probability measures on the circle $\TT^1$ such that, for every measurable
set $A$,
\begin{equation}
\mu(A)\ = \ \int_{\TT^1}\mu_\theta(A_\theta) d\theta \ .
\end{equation}
Here we use the notation
$$
A_\theta=\{x\in\TT^1\mid (\theta,x)\in A\} \ ,
$$
where $A$ is any subset of \ntorus.
The purpose of this section is to prove the following.
\begin{prop}
\label{p.blow-up-curve}
Let $R$ be a qpf circle homeomorphism and suppose there exists a curve
$\Gamma$ which satisfies the assertions of
proposition~\ref{prop.construction-curve}. Define
\begin{equation}
\label{e.xi}
\Xi \ := \ \bigcup_{n\in\Z} R^n(\Gamma)\ .
\end{equation}
Then there exist a continuous onto map $\pi$ and a homeomorphism $f$,
$$
\begin{array}{lcr}
\begin{array}{rrclc}
\pi  : &  \TT^2 & \longrightarrow & \TT^2\\
 & (\theta,x) & \longmapsto & (\theta, \pi_\theta(x)) 
\end{array}
&
 \mbox{ and } 
&
\begin{array}{crrcl}
&   f   :  & \TT^2 & \longrightarrow & \TT^2\\
 & & (\theta,x) & \longmapsto & (\theta+\omega, f_\theta(x)) 
\end{array}
\end{array}
$$
with the following properties:
\romanlist
\item for all $\theta \in \kreis$ the map $\pi_\theta$ is increasing;
\item if $\displaystyle (\theta,x)\in \Xi$
  then $\pi^{-1}(\theta,x)$ is a non-trivial interval in the circle
  $\TT^1_\theta$;
\item if $\displaystyle (\theta,x)\notin \Xi$
  then $\pi^{-1}(\theta,x)$ is a single point;
\item $\pi\circ f=R\circ \pi$.
\item $\pi^{-1}(\Xi)$ has non-empty interior.
 \listend
\end{prop}

\begin{adde} \label{add:transitivity}
  Suppose $R$ is transitive, $\Gamma$ has the additional property provided by
  addendum \ref{addendum} and in addition the set $\Xi$ is dense in
  \ntorus. Then $f$ is topologically transitive. 
  \end{adde}

Note that if $R$ is minimal, then the fact that $\Xi$ is dense is obvious. If
  $R$ is only transitive, one may construct $\Gamma$ with this property, see
  remark~\ref{r.proof-adde}~. We remark that the proof of the addendum is
  short and does not depend on the proof of the proposition (see
  section~\ref{ss.proof-adde}).
\smallskip

Theorem~\ref{t.mainthm} now follows immediately from propositions
\ref{p.blow-up-curve} and \ref{prop.construction-curve} and their
addenda. Note that the non-minimality of $f$ follows from property (v) in the
proposition, since this implies that the closure of ${\pi^{-1}(\Xi)^c}$ is a compact
invariant strict subset of \ntorus. 

\paragraph{Idea of the proof of proposition~\ref{p.blow-up-curve}}
Remember that we see proposition~\ref{p.blow-up-curve} as a generalisation of
the classical Denjoy example on the circle. Here is one way to construct the
Denjoy example.  First choose an orbit $O$ for an irrational rotation $R$ on
$\TT^1$, and let $\mu$ be a probability measure which has an atom at each
point of $O$ and no other atom. There is an (essentially unique) increasing
map $\pi : \TT^1 \ra \TT^1$ which sends the Lebesgue measure onto the measure
$\mu$. Then one looks for a circle homeomorphism $f$ such that $\pi \circ f =
R \circ \pi$: this equality determines $f$ outside $\pi^{-1}(O)$; then one
completes the construction by choosing one way to extend $f$ from the interval
$\pi^{-1}(x)$ to $\pi^{-1}(R(x))$ for all $x \in O$.

We will adapt this construction to our setting, with the following
modifications.  The role of the orbit $O$ is now played by the orbit of the
curve $\Gamma$.  The new measure $\mu$ is essentially a sum of one-dimensional
measures along the iterates of $\Gamma$ ($\mu_\theta$ has an atom at $x$ if
and only if $\displaystyle (\theta,x)$ belongs to some iterate of $\Gamma$).
As a consequence of the flat intersection hypothesis, there exists as before a
map $\pi$ sending the Lebesgue measure of $\TT^2$ to $\mu$.
Then we construct a ``nice'' measure $\nu$ on $\TT^2$ which satisfies
$\pi_*\nu=R_*\mu$. This is the difficult part of the proof, the difficulty
being linked to the fact that $\nu$ is not uniquely determined by this
equality.  Then $f$ is defined as the (essentially unique) map sending the
Lebesgue measure onto $\nu$. The equality $\pi\circ f=R\circ \pi$ will follow
automatically. The construction is summed up by the following commutative
diagram.
$$
\xymatrix{ 
\TT^2, \mathrm{Leb}  \ar@{>}[d] ^{\pi}  \ar@{>}[rr] ^{f}  
&& \TT^2 , \nu  \ar@{>}[d] ^{\pi} 
\\
\TT^2, \mu   \ar@{>}[rr] ^{R}  
&& \TT^2 , R_* \mu 
}
$$



\subsection{The semi-conjugacy~$\pi$}
\label{Pi}

In this section, we consider any sequence $(\Gamma_n)_{n\in\ZZ}$ of curves in
$\TT^2$ such that, for every $i,j\in\ZZ$, the curves $\Gamma_i$ and $\Gamma_j$
have flat intersections. We denote $\Xi := \bigcup_{n\in\Z} \Gamma_n$.

\subsubsection{Construction of the measure $\mu$}
For any graph $\Gamma$, we denote by $\delta_{\Gamma}$ the probability measure
on $\TT^2$ whose conditional $\delta_{\Gamma,\theta}$ is the Dirac mass at the
point $x$ such that $(\theta,x) \in \Gamma$.  We choose a sequence of
non-negative real numbers $(a_n)_{n\in\ZZ}$ such that $\beta:=1-\sum_{n\in\ZZ}
a_n$ is positive. Then we define a measure $\mu$ as follows:
\begin{equation} \label{eq:mudef}
  \mu\ := \ \beta\Leb + \sum_{n\in\ZZ}a_n \delta_{\Gamma_{n}}.
\end{equation}

\subsubsection{Definition of $\pi$}
The flat intersections hypothesis plays a crucial role in the addendum of the
following proposition.
\begin{prop}
\label{p.pi}
For any probability measure $\mu$ as defined by~\eqref{eq:mudef},
there exists a continuous onto map $\pi\colon\TT^2\to \TT^2$ 
of the form $(\theta,x)\mapsto (\theta,\pi_\theta(x))$, such that, for
every $\theta$, the map $\pi_\theta:\TT^1\to\TT^1$ is increasing  and
maps the Lebesgue measure onto $\mu_\theta$. 
\end{prop}
In particular, $\pi$ satisfies the properties (i), (ii) and
(iii) stated in proposition~\ref{p.blow-up-curve}.

\begin{adde} \label{add.pi}
For any given $n_0\in\Z$, it is possible to choose $\pi$ in
proposition~\ref{p.pi} such that the set $B_{n_0}=\pi^{-1}(\Gamma_{n_0})$
contains the annulus $S^1\times [0,a_{n_0}]$.
\end{adde}

\begin{proof}[Proof of proposition~\ref{p.pi} and addendum~\ref{add.pi}]
Fix an integer $n_0\in\ZZ$. Let $\Phi :
(\theta,x)\mapsto(\theta,x+\alpha(\theta))$ be the skew rotation sending the
graph $\Gamma_{n_0}$ to the zero section $\TT^1\times\{0\}$.  If we construct
a map $\pi$ satisfying the statement of the proposition and its addendum with
the graphs $\Gamma_{n}$ replaced by $\Phi(\Gamma_{n})$ and the measure $\mu$
replaced by $\Phi_* \mu$, then the map $\pi' = \Phi^{-1} \circ \pi$ will
satisfy the statements for the original objects. Thus we may change
coordinates under $\Phi$ and assume that $\Gamma_{n_0}$ is the zero section
$\TT^1\times\{0\}$.  In order to define the map $\pi:\TT^2\to\TT^2$, we will
first construct a map $\widehat\pi:\TT^1\times [0,1]\to\TT^1\times [0,1]$.
The following lemma is easy but crucial.
\begin{lemm}
\label{l.lift-measure}
Denote by $P : \TT^1\times [0,1] \to \ntorus$ the natural projection. Then there
exists a \emph{continuous lift} $\widehat\mu$ of $\mu$, that is,  
$\widehat\mu$ is a probability measure  on $\TT^1\times [0,1]$ such that
$P_*\widehat\mu=\mu$, and such that $\widehat\mu_\theta$ depends
continuously on $\theta$ (with respect to the weak topology on the space of probability measures).
\end{lemm}

\begin{proof}[Proof of lemma~\ref{l.lift-measure}]
Clearly, the Lebesgue measure on $\TT^1\times [0,1]$ is a continuous lift of
the Lebesgue measure on $\TT^2$. Seemingly, since $\Gamma_{n_{0}}$ is the null
section, the measure $\delta_{\TT^1\times\{0\}}$ in $\TT^1\times [0,1]$ is a
lift of $\delta_{\Gamma_{n_{0}}}$ in $\TT^2$.  It remains to prove that any
measure $m=\delta_{\Gamma}$ on a graph $\Gamma$ having a flat intersection
with the null section $\Gamma_{n_{0}}$ has a continuous lift $\widehat m$.
Now let $\theta$ be such that the points of $\Gamma$ and $\Gamma_{n_{0}}$ on
$\TT^1_{\theta}$ are distinct. Then we choose $\widehat m_{\theta}$ to be the
only measure that projects down to $m_{\theta}$, that is, the Dirac mass on
$P^{-1}(\Gamma_{\theta})$.  By continuity this determines the value of
$\widehat m_{\theta} \in \{\delta_{0},\delta_{1}\}$ when $\theta$ is an
endpoint of a (non trivial) interval $I$ where $\Gamma$ coincides with the
null section. Then we extend the construction on such an interval $I$ by
continuously (e. g. linearly) interpolating the Dirac masses $\delta_{0}$ and
$\delta _{1}$.
\end{proof}

Next, we define the map $\widehat\pi:\TT^1\times [0,1]\to\TT^1\times [0,1]$ by
$$
\widehat\pi(\theta,x):=(\theta,\widehat\pi_\theta(x))
$$
and
\begin{equation}
\widehat\pi_\theta(x)=\min\{y\in [0,1] \mid \widehat\mu_\theta([0,y])\geq x\}.   
\end{equation}
Observe that the map $\widehat\pi$ is continuous; indeed:
\begin{itemize}
\item the map $\widehat\pi_\theta:[0,1]\to [0,1]$ is continuous for every 
  $\theta\in\TT^1$ since the measure $\widehat\mu_\theta$ gives a
  positive mass to every open set in $\{\theta\}\times [0,1]$ (recall that
  $\beta > 0$);
\item the map $\widehat\pi_\theta$ depends continuously on $\theta$
  since the measure $\widehat\mu_\theta$ depends continuously on
  $\theta$.
\end{itemize}
Moreover, by construction, for every $\theta$ the map
$\widehat\pi_\theta$ is increasing and maps the Lebesgue measure of
$[0,1]$ onto $\widehat\mu_\theta$. 

Clearly, the map $\widehat\pi:\TT^1\times [0,1]\to\TT^1\times [0,1]$ induces a
continuous map $\pi:\TT^2\to\TT^2$ having the wanted properties.
Note that this map also satisfies the property required by addendum~\ref{add.pi}.
\end{proof}


\subsubsection{Uniqueness of $\pi$}
\begin{prop}\label{p.uniqueness-pi}
Let us consider a measure $\mu$ as defined in~\eqref{eq:mudef} on $\TT^2$
and two maps $\pi,\pi'$ which both satisfy the assertions of proposition~\ref{p.pi}.
Then, there exists a continuous skew rotation
$A:(\theta,x)\mapsto(\theta,x+\alpha(\theta))$ of $\TT^2$ such
that $\pi=\pi'\circ A$. 
\end{prop}

\begin{proof}
For every $\theta\in\TT^1$, both $\pi_\theta$ and $\pi_\theta'$ are
circle maps which map the Lebesgue measure of $\TT^1_\theta$ onto
$\mu_\theta$. It follows that, for every $\theta\in\TT^1$, the circle
maps $\pi_\theta$ and $\pi_\theta'$ coincide modulo a rotation: for
every $\theta\in\TT^1$, there exists $\alpha(\theta)\in\TT^1$ such that
$\pi_\theta(y)=\pi_\theta'(y+\alpha(\theta))$. 
This provides a skew rotation 
$A:(\theta,x)\mapsto(\theta,x+\alpha(\theta))$ of $\TT^2$ such that 
$\pi=\pi'\circ A$. 
We are left to show that $A$ is continuous, i.e. that $\alpha(\theta)$
depends continuously on $\theta$. 

Fix $\theta_0\in\TT^1$, and $\epsilon>0$. Then choose $x_0\in\TT^1_\theta$
such that $\mu_{\theta_0}(\{x_0\})=0$.  For every $\theta$, we consider the
intervals $I_\theta:=\pi_{\theta}^{-1}(\{x_0\})$ and
$I_\theta':=(\pi_{\theta}')^{-1}(\{x_0\})$.  Since
$\mu_{\theta_0}(\{x_0\})=0$, both intervals $I_{\theta_0}$ and
$I_{\theta_0}'$ consist of a single point. Moreover, the continuity of $\pi$
implies that there exists $\delta>0$ such that, for $\theta\in
[\theta_0-\delta,\theta_0+\delta]$, the interval $I_{\theta}$
(resp. $I_\theta'$) is included in an $\epsilon$-neighbourhood of the interval
$I_{\theta_0}$ (resp. $I_{\theta_0}'$). Now, for every $\theta$, the rotation
$y\mapsto y+\alpha(\theta)$ maps the interval $I_\theta$ to the interval
$I_\theta'$. Consequently $\alpha$ is continuous at $\theta_0$, and as
$\theta_0$ was arbitrary it follows that $\alpha$ is continuous on the whole
of $\kreis$. 
\end{proof}



\subsubsection{Description of the sets $\pi^{-1}(\Gamma_n)$}
In order to construct the measure $\nu$ in the next section, we will need some
more details about the geometry of the sets $B_n:=\pi^{-1}(\Gamma_n)$.  Note
that the following proposition will only be used in the construction of $\nu$,
and it may be a good idea to skip the proof for a first reading.
\begin{prop} 
\label{lem:partition}
There exists a sequence of open sets $(U_n)$, such that for all $n\in\Z$ there holds: 
\romanlist
\item $U_n \ssq \mathrm{int}(B_n)$;
\item $U_n \cap U_m = \emptyset \ \forall m \neq n$;
\item $\Leb(U_{n,\theta}) = a_n$;
\item for every $\theta$ and every $n$, $U_{n,\theta}$ is the union of at most
$2|n|+1$ open intervals.  \listend
\end{prop}

\begin{bem}~ \label{rem:partition}
\alphlist
\item Item (i) of this proposition implies property (v) of
  proposition~\ref{p.blow-up-curve}~.  
\item \label{rem:partition-b}
For every $\theta$, the union over $n$ of the sets $U_{n,\theta}$ is contained
in $(\inte(\pi^{-1}(\Xi)))_\theta$. Property (iii) implies that
this union has full measure in $(\pi^{-1}(\Xi))_\theta$. Therefore, for any
$\thx \in \Xi$, the intersection $(\inte(\pi^{-1}(\Xi)))_\theta \cap
\pi_\theta^{-1}(\{x\})$ has full measure in $\pi_\theta^{-1}(\{x\})$, and
therefore is dense in $\pi_\theta^{-1}(\{x\})$. This observation will be used
in the construction of specific examples in section~\ref{Structure}.
\listend
\end{bem}

\begin{adde}
\label{p.V}
For any $\varepsilon>0$ and any $n\in \ZZ$, there exists a compact
set $V_n\subset U_n$ that satisfies
$\Leb(V_{n,\theta}) \geq (1-\epsilon)a_n  \ \forall \theta \in \kreis$.  
\end{adde}

\begin{proof}[Proof of proposition~\ref{lem:partition}]
Let us first give an idea of the proof, by explaining the construction of
$U_{0}$ and $U_{1}$. According to addendum~\ref{add.pi}, we may assume that
$\pi^{-1}(\Gamma_0)$ contains the annulus $A_0:=\TT^1\times [0,a_0]$. We
define $U_{0}$ to be the interior of this annulus. To construct $U_{1}$, we
first note that the map $\pi$ can be factorized as $\pi = \pi_{0} \circ P_{0}$
where $P_{0}$ consists in collapsing the annulus $U_{0}$. Furthermore, the map
$\pi_{0}$ is very similar to the map $\pi$: the results concerning $\pi$ will
also apply to $\pi_{0}$, and in particular we will find an annulus $A_{1} $
included in $\pi_{0}^{-1}(\Gamma_{1})$ having the desired width. Taking the
complement of the curve $P_{0}(U_{0})$ in the interior of $A_{1}$, and
bringing this open set back under $P_{0}$, will provide the open set $U_{1}$;
note that in each fibre $U_{1}$ consists of at most two intervals, as
required.

By relabelling the sequence $\Gamma_{n}$, we can assume that it is indexed
over $\NN$ instead of $\ZZ$, which is more suitable for a proof by
induction. Formally, every $n$ in the following proof should read $\varphi(n)$
where $\varphi$ is some bijection between $\NN$ and $\ZZ$. Note in particular
that this changes property (iv), that now reads:
\begin{enumerate}
\item[\emph{(iv-bis)}] \emph{for every $\theta$ and every $n$, $U_{n,\theta}$
is the union of at most $n+1$ open intervals}.
\end{enumerate}

The construction will be done by induction on $n\in\NN$ by assuming the
following additional hypotheses.
\begin{itemize}
\item[\textit{(v)}]
\begin{enumerate}
\item \textit{There exists a fibre respecting monotonic map $P_{n}: \TT^2
 \rightarrow \TT^2$ having the following property: the restriction of the
 Lebesgue measure to the complement of $U_{0,\theta} \cup \cdots \cup
 U_{n,\theta}$ is sent to the measure $\alpha_{n}\Leb$ where $\alpha_{n}=
 1-a_{0}- \cdots - a_{n}$. }
\item \textit{The projections of the open sets $U_0,\dots,U_n$ by $P_n$ are
$n+1$ continuous graphs.}
\item \textit{There exists a fibre respecting monotonic map $\pi_n\colon
\TT^2\to \TT^2$, such that for every $\theta$, the map $\pi_{n,\theta} :
\kreis \rightarrow \kreis$ sends the Lebesgue measure on the fibre measure
$\mu'(n)_{\theta}$ of the measure
$$\mu'(n) = \frac{1}{\alpha_{n}} \left(  \beta\Leb + \sum_{i>n}a_i\mu^i  \right).$$}
\item \textit{We have $\pi_n\circ P_n=\pi$.}
\end{enumerate}
\end{itemize}
Note that the crucial hypothesis here is the fact that the simultaneous
collapsing of the sets $U_0,\dots,U_n$ yields $n+1$ continuous graphs
(hypothesis~\textit{(v)}.2 above). Hypothesis~\textit{(v)} is illustrated by
the following commutative diagram, showing the correspondence between the
different measures.
$$
\xymatrix{ 
 \mathrm{Leb} \ar@{>}[dr]   ^{P_{n}}   \ar@{>}[dd]   ^{\pi}
 \\
& \alpha_{n} \mathrm{Leb} +\sum_{i=0}^n a_{i} \delta_{P_{n}(U_{i})}
\ar@{>}[dl]   ^{\pi_{n}}
 \\
 \mu = \alpha_{n}\mu'(n) + \sum_{i=0}^n a_{i} \delta_{\Gamma_{i}}
}
$$

Assume we have constructed the sets $U_{0}, \dots , U_{n}$ and the maps $P_{n}, \pi_{n}$
satisfying the above hypotheses.


Let $q : \TT^2 \mapsto \TT^2$ be the fibre respecting increasing map that
sends the annulus $\bar A:=\TT^1\times [0,\frac{a_{n+1}}{\alpha_n}]$ on the
null section $\kreis \times \{0\}$, and sends the restriction of the Lebesgue
measure to the complement of $[0,\frac{a_{n+1}}{\alpha_n}]$ on the measure
$\left(1-\frac{a_{n+1}}{\alpha_{n}}\right)\Leb$ (see the diagram below).  Let
$\pi_{n+1} : \TT^2 \mapsto \TT^2$ be a fibre respecting increasing map that
sends the null section on the curve $\Gamma_{n+1}$, and sends the Lebesgue
measure in each fibre to the measure $\mu'(n+1)= {\alpha_{n+1}}^{-1} \left(
\beta\Leb + \sum_{i>n+1}a_i\mu^i \right)$, where $\alpha_{n+1}=1-a_{0}- \cdots
- a_{n+1}$.  The existence of such a map $\pi_{n+1}$ is guaranteed by
proposition~\ref{p.pi} and its addendum, applied with $\mu$ replaced by
$\mu'(n+1)$, and choosing $n_0=n+1$. In particular, this implies that property
\emph{(v)}.3 is satisfied for $n+1$.

$$
\xymatrix{ 
\mathrm{Leb} \ar@{-->}[rr]   ^{\Omega}   \ar@{>}[dd]   _{\pi_{n}}
&& \mathrm{Leb} = \mathrm{Leb}_{\bar A^c} + \mathrm{Leb}_{\bar A}
\ar@{>}[d]   ^{q}
 \\
 && \frac{\alpha_{n+1}}{\alpha_{n}} \mathrm{Leb} + \frac{a_{n+1}}{\alpha_{n}}\delta_{\TT^1 \times \{0\} }
  \ar@{>}[dll]   _{\pi_{n+1}}
 \\
\mu'(n)  = \frac{\alpha_{n+1}}{\alpha_{n}} \mu'(n+1) + \frac{a_{n+1}}{\alpha_{n}} \delta_{\Gamma_{n+1}}
}
$$

Now the map $\pi_{n+1} \circ q$ is fibre-respecting and increasing, and it is
 easy to see that it sends the Lebesgue measure on the measure $\mu'(n)$. The
 map $\pi_{n}$ shares the same properties.  According to
 proposition~\ref{p.uniqueness-pi}, there exists a fibered rotation $\Omega$
 such that $\pi_{n} = \pi_{n+1} \circ q \circ \Omega$.  We take $P_{n+1} = q
 \circ \Omega \circ P_{n}$.  Note that $\pi_{n+1} \circ P_{n+1} = \pi_{n}
 \circ P_{n} = \pi$, so that property \textit{(v)}.4 is satisfied for $n+1$.
 We have the following commutative diagram.
$$
\xymatrix{ 
\TT^2 
\ar@{>}[ddd]   _{\pi}
\ar@{>}[dr]   ^{P_{n}}
\ar@{>}   `[drrr]  `[ddrr]  ^{P_{n+1}}  [ddrr]
\\
&\TT^2 \ar@{>}[r]   ^{\Omega}   \ar@{>}[ldd]   ^{\pi_{n}}
& \TT^2
\ar@{>}[d]   ^{q}
&
 \\
 && \TT^2
  \ar@{>}[dll]   ^{\pi_{n+1}}
 \\
\TT^2
}
$$ Let $A$ be the inverse image by $\Omega$ of the annulus $\bar A=\TT^1\times
[0,\frac{a_{n+1}}{\alpha_n}]$. Since $\Omega$ is a fibered rotation, by
definition of $\bar A$, there exists a continuous map $\sigma : \kreis
\rightarrow \kreis$ such that
$$
A=\left\{(\theta,x), \; \sigma(\theta)\leq x \leq \sigma(\theta)+\frac{a_{n+1}}{\alpha_{n}}\right\}.
$$ By the definitions of the maps $\Omega,q,\pi_{n+1}$, the annulus $A$ is
contained in $\pi_{n}^{-1}(\Gamma_{n+1})$.  \def\inte{\mathrm{Inte}}
\def\adhe{\mathrm{Adhe}} Let $\De_0,\dots,\De_n$ be the $n+1$ continuous
graphs $P_n(U_0),\dots,P_n(U_n)$, provided by property (v).2 for $n$.  Let
$$
U_{n+1} =  P_{n}^{-1}\left(\inte(A) \setminus (\De_{0} \cup \cdots \cup \De_{n})\right).
$$ This is clearly an open set that is disjoint from $U_0,\dots, U_n$. Thus
(ii) is satisfied for $n+1$.  We now check the remaining properties
at step $n+1$.  By construction, we have $\pi_{n}(A) = \Gamma_{n+1}$, so that
$P_{n}^{-1}(A) \subset \pi^{-1}(\Gamma_{n+1}) = B_{n+1}$.  Hence, we have
(i) for $n+1$.  Note that for any $\theta$
$$\left(\inte(A) \setminus (\De_{0} \cup \cdots \cup \De_{n})\right)_{\theta}=
\inte(A_\theta)\setminus (\De_{0,\theta} \cup \cdots \cup \De_{n,\theta}).$$
Thus this set is the union of at most $n+1$ intervals.
Hence, since $P_{n+1,\theta}$ is increasing, so is the set
$$U_{n+1,\theta}=P_{n,\theta}^{-1}\left(\inte(A_\theta)\setminus
(\De_{0,\theta} \cup \cdots \cup \De_{n,\theta})\right).$$ This gives
(iv-bis).  Since $U_{n+1,\theta}$ is disjoint from
$U_{0,\theta}\cup\dots\cup U_{n,\theta}$, by the induction hypothesis
\textit{(v)}.1 for $n$ and the definition of $A$, we have
$$
\Leb(U_{n+1,\theta})=\alpha_n\Leb\left(\inte(A_\theta)\setminus
(\De_{0,\theta} \cup \cdots \cup \De_{n,\theta})\right)=\alpha_n\Leb
(A_\theta)=a_{n+1}.
$$
 This yields (iii) for $n+1$. Furthermore
$(P_{n,\theta})_* (\mbox{Leb}_{|U_{n+1,\theta}} ) =\alpha_{n}\mbox{Leb}_{|A_{\theta}}$.
Using the induction hypothesis \textit{(v)}.1 for $n$, we see that 
$$ (P_{n,\theta})_* (\mbox{Leb}_{|(U_{0,\theta} \cup \cdots \cup
U_{n+1,\theta})^c} ) = \alpha_{n}\mbox{Leb}_{|A_\theta^c}.
$$ As $\Omega_\theta$ is just a rotation, and $q_\theta$ maps
 $\mbox{Leb}_{|\bar A_\theta^c}$ to $(1-\frac{a_{n+1}}{\alpha_n})\mbox{Leb}$, the
 projection $P_{n+1}$ satisfies property \textit{(v)}.1.  By construction,
 $P_{n+1}(U_{n+1})=\TT^1\times\{0\}$, and this implies property \textit{(v)}.2
 for $n+1$.
\end{proof}

\begin{proof}[Proof of addendum~\ref{p.V}]
Let us fix $\varepsilon>0$ and $n\in \ZZ$.  For a given $\theta$, the Lebesgue
measure of $U_{n,\theta}$ is equal to $a_n$.  Hence, there exists a closed set
$V_\theta\subset \TT^1$ contained in $U_{n,\theta}$ whose Lebesgue measure is
larger than $(1-\varepsilon)a_n$.  Since $U_n$ is open, for any $\theta'$ in a
closed neighbourhood $W_\theta$ of $\theta\in \TT^1$ we still have
$V_{\theta}\subset U_{n,\theta'}$.  Consider $\theta_1,\dots,\theta_s$ such
that $\TT^1$ is covered by the $W_{\theta_1},\dots, W_{\theta_s}$.  The
proposition now holds with the set
$$V_n=\bigcup_{i=1}^s W_{\theta_i}\times V_{\theta_i}.$$
\end{proof}


\subsection{The measure $\nu$}
\label{Nu}
 
\emph{From now on, we assume that $\Gamma_{n}=R^n(\Gamma)$, where $\Gamma$ is
a curve having flat intersection with all its iterates.} Furthermore, we
consider a measure $\mu$ as defined by~\eqref{eq:mudef} and a map $\pi$ as
provided by proposition~\ref{p.pi}.

As mentioned at the beginning of this section, we construct the homeomorphism
$f$ via a measure $\nu$ on \ntorus. We will require that $\nu$ satisfies the
following properties:
\begin{itemize}
\item[$(\nu1)$]  For all $\theta \in \kreis$ the measure $\nu_\theta$ is
  continuous (has no atom) and has full support.
\item[$(\nu2)$]The mapping $\theta \mapsto \nu_\theta$ is continuous.
\item[$(\nu3)$] $\pi_*\nu \ = \ R_*\mu$.
\item[$(\nu4)$] In case $R$ is a diffeomorphism which preserves the Lebesgue
measure, $\nu$ has a continuous and positive density with respect to the
Lebesgue measure.
\end{itemize}
Note that $R_*\mu=\beta.R_*\mbox{Leb}+\sum_{n\in\ZZ} a_n
\delta_{\Gamma_{n+1}}$.

\subsubsection{Construction of $\nu$: the continuous case}
In order to construct the measure $\nu$, we first note that, if $\epsilon >
0,\ U_n$ and $V_n$ are chosen as in proposition~\ref{lem:partition}, then due
to the Lemma of Urysohn there exist continuous functions $g_n : \ntorus \ra
[0,1]$ such that \romanlist
\item[$(g1)$] \quad $g_n^{-1}(0,1] \ = \ U_n$;
\item[$(g2)$] \quad $(1-\epsilon) a_n \ \leq \ b_n(\theta) \ \leq \ a_n \ $ where $
  \ b_n(\theta) \ := \ \int_{U_{n,\theta}} g_n(\theta,x) \ dx$.  \listend
(The lower bound in $(g2)$ can be ensured just by requiring $g_{n\mid V_n} =1$.)
Now let
$$ g\thx \ := \ \sum_{n\in\Z} \frac{a_{n}}{b_{n+1}(\theta)}\cdot g_{n+1}\thx .$$
\noindent
  This positive function is in ${\cal L}^1_{\Leb}(\ntorus)$, since the ${\cal L}^1_{\Leb}$-norm
  of the $n$-th term in the sum is exactly $a_n$.
  \footnote{Note that the function $g$ is not continuous since
  $\frac{a_n}{b_{n+1}(\theta)}$ does not tend to $0$.}
For the same reason, for any $\theta$, $x\mapsto g_n(\theta,x)$ is in ${\cal L}^1_{\Leb}(\kreis)$
and has a norm bounded by $a_n$. For each $n$ the mapping $$\kreis \to {\cal L}^1_{\Leb}(\kreis) \quad ,
\quad \theta \mapsto g_n(\theta,\ldot)$$ is continuous, the same is true for
$\theta \mapsto g(\theta,\ldot)$ (as the uniform limit of a sequence of
continuous functions from $\TT^1$ to ${\cal
L}^1_{\Leb}(\kreis)$). Consequently, if we let $\nu^1 := g\Leb$, then $\theta
\mapsto \nu^1_\theta$ is continuous with respect to the topology of weak
convergence.  Due to $(g1)$ the function $x \mapsto g\thx$ is strictly
positive on $\bigcup_{n\in\Z} U_{n,\theta}$, and due to (iii) in
proposition \ref{lem:partition} this set is dense in $B_\theta$ (remember that
$B = \pi^{-1}(\Xi)$). Thus $\nu^1_\theta$ has full support in $B_\theta$. By
construction, the measure $\nu^1$ projects to $\sum_{n\in\Z} a_n
\delta_{\Gamma_{n+1}}$.

Further, as $\pi$ is injective on $B^c = \pi^{-1}(\Xi^c)$ and $(R_*\Leb)(\Xi)
= 0$, the measure $\nu^2$ defined by $\nu^2(A) := \beta(R_*\Leb)\circ\pi(A)$
is well-defined and obviously projects to $\beta(R_*\Leb)$. In addition,
$\theta \mapsto \nu^2_\theta$ is continuous, and the fibre measures
$\nu^2_\theta$ are continuous and have full support on $\mathrm{int}(B^c_\theta)$. Altogether,
this implies that
\[
   \nu \ := \ \nu^1 + \nu^2
\]
satisfies $(\nu1)$--$(\nu3)$. 

\subsubsection{Construction of $\nu$: the differentiable case}
\label{Nu-continuous} As above
we suppose that $\epsilon,U_n$ and $V_n$ are chosen as in
propositions~\ref{lem:partition} and~\ref{p.V}. Further, we assume that $R$ is
a diffeomorphism which preserves Lebesgue measure. Note that $\pi$ projects
$\Leb_{|B^c}$ to $\beta \Leb$, and $\Leb_{|U_n}$ to $a_n
\delta_{\Gamma_n}$. Since $R_*\mu=\beta\mbox{Leb}+\sum_{n\in\ZZ} a_n
\delta_{\Gamma_{n+1}}$, we will construct the measure $\nu = h\mbox{Leb}$
by defining a continuous density $h$ which satisfies $h_{|B^c} \equiv 1$ and
\begin{equation} \label{eq:hcond}
  \int_{U_{n+1,\theta}} h(\theta,x) \ dx \ = \ a_n \quad \forall \theta \in
  \kreis,\ n \in \Z \ .
\end{equation}
This is done by 
\begin{equation} \label{eq:hdef}
  h\thx \ := \ 1 - \sum_{n\in\Z} (a_{n+1} - a_n) \cdot
  \frac{g_{n+1}\thx}{b_{n+1}(\theta)} \ .
\end{equation}
It is obvious by construction that $h$ satisfies (\ref{eq:hcond}) and
therefore projects to $R_*\mu$.

 It remains to show that, for a suitable choice of $a_n$ in (\ref{eq:mudef})
and $\epsilon$ in proposition~\ref{lem:partition}, the function $h$ is continuous
and strictly positive, such that $(\nu4)$ (and thus also $(\nu1)$ and
$(\nu2)$) hold. In order to see this, recall that $\| g_n \|_{{\cal C}^0} =
1$. Further 
\begin{equation} \label{eq:andiff}
\left|\frac{a_{n+1} - a_n}{b_{n+1}(\theta)}\right| \ \leq \
\frac{|a_{n+1}-a_n|}{(1-\epsilon) a_n} \ ,
\end{equation}
and for a suitable choice of the sequence $a_n$ (fixing any $\varepsilon\in
(0,1)$ one can choose $a_n = (|n|+k)^{-2}$ for sufficiently large $k$),
we have $\sum_{n\in\Z} a_n < 1$, the
right side of (\ref{eq:andiff}) is strictly smaller than 1 and converges to 0
as $n \to \infty$. Since all the $g_{n}$ have disjoint support, this implies
that $h$ will be continuous and strictly positive as required. Further, if $k$
is large enough we also have $\sum_{n\in\Z} a_n < 1$.

\subsubsection{${\cal C}^{1+\alpha}$-Examples} \label{Nu-Hoelder}
 Without going too much into
detail, we remark that at least in the case where $R$ leaves the Lebesgue
 measure invariant, as in the case of an irrational translation of the torus,
 it is possible to obtain examples with ${\cal C}^{1+\alpha}$ fibre maps for
 any $\alpha \in (0,1/2)$. For this, it suffices that the density $h$
 constructed above is $\alpha$-H\"older continuous with respect to $\theta$.
  
Although this is not explicitly stated in lemma~\ref{lem:partition}, it is
obvious from the construction of the sets $U_n$ in the proof of this lemma
that the functions $\thx \mapsto d(x,U_{n,\theta}^c)$ are continuous. In
addition, due to (iv) in proposition~\ref{lem:partition}, the sets
$U_{n,\theta}$ consist of at most $2| n| +1$ connected components on each
fibre. Therefore, the functions
\[
g_n\thx \ := \ \min\left\{1,\left(\frac{4| n|+2}{\epsilon
a_n}d(x,U_{n,\theta}^c)\right)^\alpha\right\}
\]
satisfy $(g1)$ and $(g2)$. In addition, they are $\alpha$-H\"older continuous
with respect to $\theta$ with H\"older-constant $((4| n| +2)/(\epsilon a_n))^\alpha$. In
(\ref{eq:hdef}), the functions $g_{n+1}$ are multiplied with a factor $\leq
|a_{n+1}-a_n|/(1-\epsilon)a_n$, such that the resulting product has
$\alpha$-H\"older constant $\leq \frac{((4|
n|+2)/\epsilon)^\alpha(a_{n+1}-a_n)}{(1-\epsilon)a_n^{1+\alpha}}$. Hence, if
we choose $(a_n)_{n\in\Z}$ such that
\begin{equation} \label{eq:ancond}
 \sup_{n\in\Z} \frac{| n|^\alpha | a_{n+1}-a_n|}{a_n^{\alpha+1}} \ < \ \infty \ , 
\end{equation}
then the resulting sum in (\ref{eq:hdef}) is $\alpha$-H\"older continuous
(since the functions $g_n$ all have disjoint
support). However, (\ref{eq:ancond}) is true for any sequence $a_n = |n +
k|^{-s}$ with $k \geq 1$ and $s \in (1,1/\alpha-1)$.

Of course, one should expect that the construction works for all $\alpha \in
(0,1)$. For this, one would need to show that the sets $U_n$ in
proposition~\ref{lem:partition} can be chosen such that the number of
connected components is bounded by a constant independent of $n$. Since this
would make the proof much more complicated, we refrained from doing so. 


\subsection{The homeomorphism~$f$}

In this section, we consider a projection $\pi$ as provided by
proposition~\ref{p.pi} and a measure $\nu$ which satisfies $(\nu1)$--$(\nu4)$,
as constructed in the previous sections.  We denote by $\gamma_{0},\gamma_{1}:
\TT^1 \ra \TT^1$ the maps whose respective graphs are $\Gamma_{0}$ and
$\Gamma_{1}$, and define the two (discontinuous) maps $\varphi_{0}^-,
\varphi_{1}^-$ by letting $\varphi_{i}^-(\theta)=
\mbox{Inf}(\pi^{-1}(\Gamma_{i}))$ (this has a meaning since $\TT^1$ is an
oriented circle).

\paragraph{Definition of $f$ and verification of  $\pi \circ f = R \circ \pi$.}
 Due to $(\nu1)$, the map
\[
  \begin{array}{ccccc}\eta_\theta &: &\kreis \ra \kreis &,& x \mapsto
  \nu_\theta[\varphi^-_1(\theta),x] \ \bmod 1\end{array}
\]
is a homeomorphism. Using this fact, we let
\begin{equation} 
  f_\theta(x) \ := \ \eta_{\thom}^{-1}(\Leb[\varphi_0^-(\theta),x])
\end{equation}
In other words, we simply define $f$ by requiring that it maps $\varphi_0^-$
to $\varphi^-_1$ and sends the Lebesgue measure to $\nu$, so that
\begin{equation}\label{eq:f-nu}
  \nu_{\thom}[\varphi_1^-(\thom),f_\theta(x)] \ = \
  \Leb[\varphi_0^-(\theta),x] \ .
\end{equation}
For arbitrary $x_1,x_2\in \kreis$, we obtain
\begin{equation} \label{eq:f-nuII}
  \nu_{\thom}[f_\theta(x_1),f_\theta(x_2)] \ = \ \Leb[x_1,x_2] \ .
\end{equation}
As all fibre maps $f_\theta$ are circle homeomorphisms, the map $f$ is
bijective. 

We remark that in the case where the fibre measures $\nu_\theta$ have
continuous ($\alpha$-H\"older-continuous) densities, as in sections
\ref{Nu-continuous} and \ref{Nu-Hoelder}, it follows immediately from the
definition that the fibre maps $f_\theta$ are ${\cal C}^1$ (${\cal
C}^{1+\alpha}$). \bigskip

In order to verify the relation $\pi \circ f = R \circ \pi$, it suffices to
check that for every $z=\thx\in\ntorus$ there holds
\begin{equation} \label{e.commuting-property}
f ( \pi^{-1}(z)) = \pi^{-1}(R(z)).
\end{equation}
Writing $\pi_\theta^{-1}(x) = [\xi^-_0,\xi^+_0]$ and
$\pi_{\thom}^{-1}(R_\theta(x)) = [\xi^-_1,\xi^+_1]$,
(\ref{e.commuting-property}) becomes $f_\theta(\xi^\pm_0) = \xi_1^\pm$.
 We claim that
\begin{equation} \label{e.commutingII}
\nu_{\thom}[\varphi^-_1(\thom),f_\theta(\xi^\pm_0)] \ =
\nu_{\thom}[\varphi^-_1(\thom),\xi_1^\pm] \ .
\end{equation}
Since the measure $\nu_{\thom}$ has full support, this will imply
(\ref{e.commuting-property}). To check~(\ref{e.commutingII}) for the left endpoints, we compute
\begin{eqnarray*}
 \lefteqn{\nu_{\thom}[\varphi^-_1(\thom),f_\theta(\xi^-_0)] \ \stackrel{(\ref{eq:f-nu})}{=} } \\
&  & \Leb[\varphi^-_0(\theta),\xi^-_0] \  \
 = \ \ \mu_\theta[\gamma_0(\theta),x) \\ &  = &
 (R_*\mu)_{\thom}[\gamma_1(\thom),R_\theta(x)) \ \stackrel{(\nu 3)}{=} \
 \nu_{\thom}[\varphi^-_1(\thom),\xi^-_1] \ .
\end{eqnarray*}
The argument for the right endpoints is similar.


\paragraph{Continuity of $f$.} We first show that $f$ is continuous in
$(\theta,x)$ if $(\theta,x) \notin B$. Recall that this means that
$\pi(\theta,x)$ is not contained in $\Xi=\bigcup_{n\in\Z} R^n(\Gamma)$, and
consequently the same is true for the point $z=R\circ\pi(\theta,x)$. Therefore
$z=R\circ\pi(\theta,x)$ has a unique preimage under $\pi$. As $R\circ\pi =
\pi\circ f$, this preimage is $f\thx$.  Since $\pi$ is continuous and due to
compactness, it follows that for any point $z'$ which is sufficiently close to
$z$, the preimage $\pi^{-1}(\{z'\})$ is contained in a small neighbourhood of
$f\thx$. Now suppose $(\theta',x')$ is close to $\thx$. Then
$z':=R\circ\pi(\theta',x')$ is close to $z$ by continuity. As $f(\theta',x')$
must be contained in $\pi^{-1}(\{z'\})$, this shows that $f(\theta',x')$ is
close to $f\thx$. Hence $f$ is continuous in \thx.

Now let $\thx\in\ntorus$ be arbitrary. The set $B^c$ intersects every fibre, since
this is obviously true for its image $\Xi^c$ under $\pi$. Consequently, there
exists a continuity point $(\theta,x_0)$ of $f$ on the same fibre. Suppose
that $(\theta',x')$ is a point close to \thx. Then (\ref{eq:f-nuII}) implies
that $\nu_{\theta'+\omega}([f_{\theta'}(x_0),f_{\theta'}(x')])$ is close to
$\nu_{\theta+\omega}([f_{\theta}(x_0),f_{\theta}(x)])$. As $f$ is continuous
in $(\theta,x_0)$, the points $f_{\theta'}(x_0)$ and $f_{\theta}(x_0)$ are
close. Therefore, the fact that $\nu_{\theta+\omega}$ has full support
and the continuity of $\theta\mapsto\nu_\theta$ imply that
$f_{\theta'}(x')$ has to be close to $f_{\theta}(x)$. Thus, $f$ is continuous
on all of \ntorus.
\bigskip

\subsection{Proof of addendum~\ref{add:transitivity}}
\label{ss.proof-adde}

Suppose that $R$ is transitive, $\Gamma$ has the additional property provided
by addendum \ref{addendum} and $\Xi$ is dense.  We have to show that given any
two open sets $U,V \ssq \ntorus$ there exists some $n\in\N$ such that $f^n(U)
\cap V \neq \emptyset$.

First, we claim that the interior of $B= \pi^{-1}(\Xi)$ is dense. For suppose
that there exists an open set $W \ssq B^c$. For any point $z \in W$, since
$\pi^{-1}\{z\}$ is a singleton, compactness and continuity of $\pi$ imply
that there exists some neighbourhood $\tilde W$ of $\pi(z)$, such that
$\pi^{-1}(\tilde W) \subset W$.  However, as $\tilde W$ is clearly disjoint
from $\Xi$, this contradicts the assumptions.

Thus, by reducing both sets further if necessary, we can assume that $U$ and
$V$ are two small rectangles which are included in $B$.  In fact, we can even
restrict to the case where $U \ssq B_k$ and $V \ssq B_l$ for some $k,l \in
\Z$, and that both $I:= p_1(U)$ and $J:= p_1(V)$ are intervals. However, due
to the choice of $\Gamma$ there exists some $n\in\N$ such that
$R^n(\Gamma_{k\mid I})$ crosses $\Gamma_{l\mid J}$ over $I+n\omega \cap J$. As
all the maps $\pi_\theta$ are order-preserving this implies that $f^n(U)$ has
to cross $V$ over $I':=(I +n\omega) \cap J$ (more precisely: if $O$ is chosen
as in definition~\ref{def:flatintersections}, then $\pi^{-1}(I' \times O)
\setminus V$ consists of two connected components, and $f^n(U)$ intersects
both of them). However, this is only possible if $f^n(U)$ intersects $V$.

\begin{bem} \label{r.proof-adde}
If $R$ is transitive, it is always possible to construct a curve $\Gamma$
such that $\Xi$ is dense, as required by addendum~\ref{add:transitivity}. For
this, it suffices to choose each $\Gamma_n$ in the construction of $\Gamma$ in
section~\ref{section.curve}, such that it contains a point with dense
orbit. This is possible due to remark~\ref{r.perturbation}. It follows that
for each $n\in\N$ there exists an integer $N_n$, such that the first $N_n$
iterates of $\Gamma_n$ are $1/n$-dense in $\ntorus$. If all subsequent
perturbations are chosen small enough, then for all $k\geq n$ the first $N_n$
iterates of $\Gamma_k$ will be $2/n$-dense. In the limit, this gives the
required property.
\end{bem}

\section{The minimal set} \label{Structure}

In this section, we collect some general results which concern the properties
and structure of the minimal sets of a transitive but non minimal qpf circle
homeomorphism. It is known that in the absence of invariant strips such a set
must have a complicated structure:

\begin{prop}[{\cite[theorem 4.5 and lemma 4.6]{jaeger/keller:2006}}]
\label{prop.jeakel}
  Suppose $f\in{\cal F}$ has no invariant strips and is transitive, but not
  minimal. Then any minimal invariant set $M$ which is a strict subset of
  \ntorus\ has the following two properties: 
  \alphlist
\item Every connected component $C$ of $M$ is a vertical segment, i.e.\
  $\#p_1(C) = 1$. 
\item For every open set $U\ssq \ntorus$, the set $p_1(U\cap M)$ is either
  empty or it contains an interval. 
\listend
\end{prop}

\subsection{Uniqueness and structure}\label{ss.uniqueness-structure}

We start by giving two criteria for the uniqueness of the minimal set. (For a
previous partial result on quasiperiodic $\mbox{SL}(2,\R)$-cocycles,
see~\cite[section 4.17]{herman:1983}.)

\begin{prop}
Suppose $f\in {\cal F}$ has no invariant strip, or is transitive.
Then there is only one $f$-invariant minimal set.
\end{prop}
\begin{proof}
We choose an orientation on the circle, so that the segment $(x,y)$ is well
defined whenever $x,y \in \kreis$.  Suppose that there exist two minimal sets
$K\neq K'$. Then $K \cap K' = \emptyset$.  We define the set $U_{1}$ as the
union of all vertical segments $\{\theta\}\times (x,y)$ which are disjoint
from $K\cup K'$ and satisfy $(\theta,x)\in K$ and $(\theta,y)\in K'$.
Similarly we define $U_{2}$ as the union of all vertical segments
$\{\theta\}\times (x,y)$ which are disjoint from $K\cup K'$ and satisfy
$(\theta,x)\in K'$ and $(\theta,y)\in K$.  These two sets are clearly disjoint
and intersect each fibre $\{\theta\}\times \kreis$.  As all fibre maps
$f_{\theta}$ are order preserving, they are also invariant.  Let us prove that
they have non-empty interior. The mappings $\theta\mapsto K_{\theta}$ and
$\theta\mapsto K'_{\theta}$ are semi-continuous, hence their sets of
continuity points are two dense $G_{\delta }$-sets; in particular they have a
common point $\theta_{0}$.  It is easy to see that all the points
$(\theta_{0},x)$ in $U_{1}$ belong to the interior of $U_{1}$, similar for
$U_{2}$.  Thus $\inte(U_{1})$ and $\inte(U_{2})$ are two disjoint non-empty
open invariant sets, and $f$ is not transitive.

It remains to consider the case when $f$ has no invariant strip and is not transitive.
By theorem~\ref{thm:poincare},  $f$ is semi-conjugate to an irrational torus rotation.
Therefore, due to a result by Furstenberg (see~\cite[theorem 4.1]{furstenberg:1961}),
$f$ is uniquely ergodic, hence it has a unique minimal set.
\end{proof}

Concerning the number of connected components in each fibre, we have the
following:
\begin{prop}\label{prop.c-of-K}
Let $f$ be a qpf circle homeomorphism, and $K$ a minimal set for $f$.
Let $c(\theta) \in \NN \cup\{\infty\}$ be the number of connected components of $K_{\theta}$.
Let $c(K)= \inf_{\theta \in \kreis} c(\theta)$.
Then $c(\theta)=c(K)$ on a dense  $G_{\delta}$ subset of the circle.
\end{prop}
\begin{proof}
Let $\Theta_{1}$ be the set of continuity points of the semi-continuous map
 $\theta \mapsto K_{\theta}$: this is a dense $G_{\delta }$ subset of the
 circle.  The restriction of the map $\theta \mapsto c(\theta)$ to
 $\Theta_{1}$ is again semi-continuous, so it admits a continuity point
 $\theta_{0}$.  By continuity there exists an open neighbourhood $U$ of
 $\theta_{0}$ such that $c$ is constant on $U \cap \Theta_{1}$. Since $c$ is
 invariant under the circle rotation $\theta \mapsto \theta + \omega$, it is
 constant on a dense open subset $\Theta_{2}$ of $\Theta_{1}$, and
 $\Theta_{2}$ is again a dense $G_{\delta }$ subset of the circle.  Let
 $c_{0}$ be the value of $c$ on $\Theta_{2}$, and let us prove that $c_{0} =
 c(K)$.  For this consider a fibre $\theta_{0} \in \Theta_{2}$. Since
 $\theta_{0}$ is a continuity point of $\theta \mapsto K_{\theta}$, there
 exists a neighbourhood $U$ of $\theta_{0}$ on which $c(\theta) \geq c_{0}$.
 Using the invariance of $c$ under the circle rotation, we see that this
 inequality holds on the whole circle.
\end{proof}

 This result raises the following question: \emph{for the unique minimal set
 $K$ of a non minimal qpf circle map with no invariant strip, can $c(K)$ be
 finite ?}

\subsection{Examples} 
\label{ss. examples}

\begin{prop}~ \label{p.examples}
\begin{enumerate}
\item \emph{There exists a transitive non minimal qpf circle homeomorphism
whose minimal set $K$ is a Cantor set and its intersection with any fibre
$\{\theta\}\times \kreis$ is uncountable; in particular $c(K)=+\infty$.}
\item \emph{There exists a transitive non minimal qpf circle homeomorphism
whose minimal set contains a vertical segment.}
\end{enumerate}
\end{prop}

\begin{proof}[Proof of item~1]
Choose a minimal qpf circle homeomorphism $R$, and let $\Gamma$ be a
continuous graph constructed by proposition~\ref{prop.construction-curve} and
addendum~\ref{addendum}.  Let $f$ be a qpf circle homeomorphism given by
proposition~\ref{p.blow-up-curve}, which in particular satisfies $\pi\circ
f=R\circ \pi$ and is transitive but not minimal.  Let $\Xi=\bigcup_{n\in
\ZZ}R^n(\Gamma)$. Then we claim that the unique minimal set $K$ for $f$ is the
set $\overline{\pi^{-1}(\Xi^c)}$: indeed, since $R$ is minimal $\pi(K) =
\TT^2$, thus $K$ contains $\pi^{-1}(\Xi^c)$ by (iii) of
proposition~\ref{p.blow-up-curve}.

To construct the first example we suppose that $R$ has no invariant strip and
that for each $\theta$, the set $\Xi_{\theta}$ is dense in the circle: in
particular this is always true if $R$ is a irrational rotation.  It follows
from items (ii) and (iii) of proposition~\ref{p.blow-up-curve} together with
remark~\ref{rem:partition}(b) that the set $\inte(\pi^{-1}(\Xi))$ intersects
each fibre in an open dense set. In particular, $K_{\theta}$ has empty
interior.  Since $\Xi_{\theta}^c$ is uncountable, so is $K_{\theta}$.  Since
$R$ has no invariant strip, neither has $f$. Hence every connected component
of $K$ is included in a fibre (proposition~\ref{prop.jeakel}); consequently
$K$ is totally disconnected.  Since $K$ is a minimal infinite set, it is
perfect, so $K$ is a Cantor set.
\end{proof}

\begin{proof}[Sketch of proof of item~2]
To construct the second example, we start with an irrational rotation $R_{0}$
and choose a curve $\Gamma_{0}$ as provided by
proposition~\ref{prop.construction-curve} and its addendum. We consider a
point $(\theta_{0},x_{0})$ whose orbit by $R_{0}$ is disjoint from
$\Gamma_{0}$. Applying Rees construction in~\cite{rees:1979,becroler:2006}
allows to build a minimal fibered homeomorphism $R$ that is semi-conjugate to
$R_{0}$ by a semi-conjugacy $\Phi$ such that $\Phi^{-1}(\theta,x)$ is a
nontrivial vertical segment if $(\theta,x)$ belongs to the orbit $\cal O$ of
$(\theta_{0},x_{0})$ under $R_{0}$ and a single point otherwise. Since
$\Phi^{-1}(\cal O)$ has empty interior and $R_{0}$ is minimal, $R$ is minimal.
Now we let $\Gamma = \Phi^{-1}(\Gamma_{0})$; $\Gamma$ is a continuous graph
whose iterates are disjoint from the non-trivial vertical segment $I
=\Phi^{-1}(\theta_{0},x_{0})$.  Furthermore, one can check that $R$ and
$\Gamma$ still satisfy the conclusions of
proposition~\ref{prop.construction-curve} and its addendum.  We now apply
proposition~\ref{p.blow-up-curve}. Thus we get a map $f$ whose minimal set $K$
contains the non-trivial vertical segment $\pi^{-1}(I)$.
\end{proof}

\subsection{Ergodic measures} 
\label{ErgodicMeasures}

In this subsection, we briefly want to discuss, in a rather informal way, the
consequences of our construction in sections~\ref{section.curve} and
\ref{Blow-up-curve} for the structure of the invariant measures of the
system. First, we recall an old result by Furstenberg \cite{furstenberg:1961},
which may be seen as a measure-theoretic counterpart to
theorem~\ref{thm:poincare}. In order to state it, we note that $p_1$ maps any
ergodic invariant probability measure $\mu$ of a qpf circle homeomorphism $f$
to the Lebesgue measure on $\kreis$, since this is the only invariant
probability measure for the underlying irrational circle rotation. Consequently,
$\mu$ can be decomposed as
$$
\mu(A) \ = \ \int_{\kreis} \mu_\theta(A_\theta) \ d\theta \ ,
$$ where the probability measures $\mu_\theta$ are the conditional measures of
$\mu$ with respect to the $\sigma$-algebra $p_1^{-1}({\cal B}(\kreis))$.
\begin{thm}[{Furstenberg, \cite[theorem 4.1]{furstenberg:1961}}] \label{t.furstenberg}
  Either $f$ is uniquely ergodic and for almost every $\theta$ the measure
  $\mu_\theta$ is continuous, or there exists $n\in\N$ such that for almost every
  $\theta$ the measure $\mu_\theta$ is the equidistribution on $n$ points.
\end{thm}
In the second case, one obtains a measurable invariant set $\Phi := \{\thx \mid
\mu_{\theta}\{x\} > 0\}$, which contains exactly $n$ points in every
fibre.%
\footnote{{\em A priori}, this only holds on almost every fibre. However, by
  modifying $\mu_\theta$ on a set of measure zero, one one can always assume
  without loss of generality that it holds on every fibre.}
Since such a set can always be represented as the graph of a measurable
$n$-valued function $\varphi$, one speaks of an {\em invariant
graph}. Conversely, every invariant graph determines an ergodic invariant
measure $\mu_\Phi$, given by
$$
\mu_\Phi(A) \ = \ \int_{\kreis} \frac{\#(A_\theta \cap \Phi_\theta)}{n}\ d\theta \ .
$$ (We remark that the requirement that this measure is ergodic is part of the
definition of an invariant graph.) For a more detailed discussion of these
concepts, see \cite[section 2]{jaeger/keller:2006}. 

The important fact in our context is that the two alternative cases of
theorem~\ref{t.furstenberg} are preserved by topological extension: if $f$ is
semi-conjugate to $R$ via a fibre-respecting semi-conjugacy $\pi$, then $\pi$
projects $f$-invariant graphs to $R$-invariant graphs. Conversely, the
preimage of any $R$-invariant graph under $\pi$ intersects every fibre in
exactly $n$ connected components, and the endpoints of the latter constitute
invariant graphs for the topological extension $f$.
\medskip

In order to describe how our construction affects the invariant measures, we
place ourselves in the situation of proposition~\ref{p.blow-up-curve} and
consider the two cases in theorem~\ref{t.furstenberg}~. 
\smallskip

1) When the original transformation $R$ has a unique invariant measure with
continuous fibre measures, then nothing much happens. The topological
extension $f$ will still have a unique invariant measure with continuous fibre
measures. The only effect is that the new invariant measure does
not have full topological support in \ntorus\ (since $f$ is not
minimal).
\smallskip

The question of what happens with an invariant graph $\Phi$ when passing
from $R$ to $f$ in proposition~\ref{p.blow-up-curve} (and hence in
theorem~\ref{t.mainthm}) mainly depends on the value of
$\mu_\Phi(\Gamma)$. Thus, we have to distinguish two sub-cases.
\smallskip~

2a)  When $\mu_\Phi(\Gamma)=0$, nothing changes either. The
preimage of $\Phi$ under $\pi$ constitutes an invariant graph for the new map
$f$, and the two systems $(R,\mu_\Phi)$ and $(f,\mu_{\pi^{-1}\Phi})$ are
isomorphic in the measure-theoretic sense. 

2b) The more interesting case is the
one where $\mu_\Phi(\Gamma)>0$. In this case, since $\mu_\Phi$ is ergodic and
$\Xi = \bigcup_{n\in\Z} R^n(\Gamma)$ is invariant, we have $\mu_\Phi(\Xi) =
1$. Consequently, the preimage $\pi^{-1}\thx$ of almost every point in $\Phi$
is a vertical segment. The endpoints of these segments constitute two distinct
invariant graphs for $f$, such that the invariant measure $\mu_\Phi$ has been
(at least) doubled. {\em A more difficult question, which we have to leave open
here, is the one whether there exist further invariant graphs in the preimage
of $\Phi$.}
\medskip

Without going into detail, we remark that both conditions $\mu_\Phi(\Gamma)=0$
and $\mu_\Phi(\Gamma) > 0$ can be ensured by adapting the construction of the
curve $\Gamma$ in the proof of proposition~\ref{prop.construction-curve}~. In
the former case, the crucial fact is that for any $\epsilon>0$, it is always
possible to render any segment of a given curve disjoint from $\Phi$ on a set
of fibres of measure arbitrarily close to 1 by an
$\epsilon$-perturbation. Performing an infinite sequence of smaller and
smaller perturbations, this shows that in any arbitrarily small box there
exist continuous curve segments which intersect $\Phi$ only on a set of fibres
of measure zero. Using this fact appropriately in each step of the
construction, namely when the modifications are constructed in the proof of
perturbation lemma \ref{l.perturbation} (compare also
remark~\ref{r.perturbation}), allows to ensure that the limit curve obtained
in proposition~\ref{prop.construction-curve} satisfies $\mu_\Phi(\Gamma) = 0$.

Conversely, in order to ensure $\mu_\Phi(\Gamma)>0$, it suffices to start the
construction in section~\ref{section.curve} with a curve $\Gamma_0$ which
intersects $\Phi$ on a set of positive measure. This is always possible, since
there exists compact sets $K \ssq \kreis$ of measure arbitrarily close to 1
with the property that the restriction of $\varphi$ to $K$ is continuous
(where $\varphi$ is the measurable function $\kreis\ra\kreis$ with graph
$\Phi$ as above). If the modifications in each step of the construction are
then performed only over sets of sufficiently small measure, then the
resulting limit curve $\Gamma$ will still intersect $\Phi$ on a set of
positive measure. It is even possible to ensure this condition for any finite
number of invariant graphs at the same time.

Finally, we want to mention that it is possible to repeat this ``doubling
procedure" as many times as wanted, producing a sequence of topological extensions
$f_1,f_2,\ldots$ of $R$ with more and more invariant graphs in the preimage of
an initial $R$-invariant graph. Furthermore, the ${\cal C}^0$-distance between
both $f_n$ and $f_{n+1}$ and the corresponding semi-conjugacies can be made
arbitrarily small in each step. Then the $f_n$ converge to a limit $f_\infty$,
which is a topological extension of $R$ with an infinite number of invariant
graphs that project down to $\Phi$.

\subsection{The linear case}
Finally, we restrict ourselves to qpf linear circle homeomorphisms.  We
identify the $2$-torus $\TT^2$ with $\kreis \times \PP^1(\RR)$ and consider
the projective action of $\mbox{SL}(2,\RR)$ on $\PP^1(\RR)$.  Then a \emph{qpf
linear circle homeomorphism} is a homeomorphism of the $2$-torus, isotopic to
the identity, of the form $(\theta,x) \mapsto (\theta + \omega,f_{\theta}(x))$
where $f_{\theta} \in \mbox{SL}(2,\RR)$.

\begin{prop} 
\label{p.twopoints}
Let $f$ be a qpf linear circle homeomorphism. Assume $f$ is transitive but not
minimal.  Let $K$ be the unique minimal invariant set for $f$.  Then
$K_{\theta}$ contains at most two points for every $\theta$ in a dense
$G_{\delta}$ subset of the circle. In particular we have $c(K) = 1$ or $2$ in
proposition~\ref{prop.c-of-K}.
\end{prop}

In case there is no invariant strip, the preceding result can be improved further. Let us recall that we know no example of a qpf linear circle homeomorphism with no invariant strip which is not minimal (so that the following statement could turn out to be void).

\begin{prop} 
\label{p.sltr-onepoint}
Let $f$ be a qpf linear circle homeomorphism. Assume $f$ is not
minimal and has no invariant strip. Let $K$ be the unique minimal invariant
set for $f$. Then $K_{\theta}$ is a singleton for every $\theta$ in a dense
$G_{\delta}$ subset of the circle. In particular we have $c(K) = 1$ in
proposition~\ref{prop.c-of-K}.
\end{prop}

In the case where $f$ has invariant strips, one has the following result:

\begin{prop}[\cite{bjerkloev/johnson:2007}]
\label{p.bjerkloev/johnson}
Let $f$ be a qpf linear circle homeomorphism. Assume $f$ has invariant strips and let $K$ be a minimal invariant set for $f$. Then there are two (non exclusive) possibilities:
\begin{enumerate}
\item  $K$ is a continuous $(p,q)$-invariant graph, or
\item for every $\theta$ in a dense $G_{\delta}$ subset of the circle
the cardinality of the set $K_\theta$ is $1$ or $2$.
In particular we have $c(K) = 1$ or $2$ in proposition~\ref{prop.c-of-K}.
\end{enumerate}
\end{prop}

This result is in \cite{bjerkloev/johnson:2007} formulated for continuous time systems,
but the translation to the discrete time case and vice versa is plain sailing.
We provide a new proof of it. The authors in \cite{bjerkloev/johnson:2007} also show that any minimal set belongs to one of five different cases, but for one of the possibilities
(which they call ``Denjoy extension'') they leave open whether it can be realised or not.
The proposition~\ref{p.twopoints} above excludes the existence of such Denjoy extensions.

Proposition~\ref{prop:cocycles} now follows from propositions~\ref{p.sltr-onepoint} and~\ref{p.bjerkloev/johnson}:  given a strict minimal set $K$ for a qpf linear circle homeomorphism $f$, if $f$ has no invariant strip then for a generic $\theta$ the cardinality of $K_{\theta}$ is $1$ (proposition~\ref{p.sltr-onepoint});
if $f$ has invariant strips, then the cardinality of $K_{\theta}$ is generically $1$ or $2$, or $K$ is a $(p,q)$-invariant graph (proposition~\ref{p.bjerkloev/johnson}). Note that  if the cardinality of $K_{\theta}$ is generically greater than $2$ then we are in the last case with $pq >2$, and then it is easy to see that $f$ is actually conjugate to a rotation.

\medskip

\begin{proof}[Proof of proposition~\ref{p.twopoints}]
Let $f$ and $K$ be as in the proposition.  Let $\Theta_{1}$ be the set of
 $\theta_{0} \in \kreis$ such that $\theta_{0}$ is a continuity point of
 $\theta \mapsto K_{\theta}$, and such that for a dense set of $x_{0} \in
 \kreis$ the positive orbit of $(\theta_{0},x_{0})$ is dense in $\TT^2$. Given the transitivity of $f$,
 the set of points whose positive orbit is dense is a dense $G_{\delta}$
 subset of the $2$-torus.  Thus $\Theta_{1}$ contains a dense $G_{\delta }$
 subset of the circle.%
\footnote{Recall that for any dense $G_{\delta }$ subset $E$ of $\TT^2$, there
exists a dense $G_{\delta }$ subset of points $\theta \in \TT^1$ for which
$E_{\theta}$ is a dense $G_{\delta }$ subset in $\kreis$.}

We now suppose that there exists some $\theta_{0} \in \Theta_{1}$ such that $K_{\theta_{0}}$ is not
a single point.  We choose a connected component $(a,b)$ of $\kreis \setminus
K_{\theta_{0}}$; note that $a \neq b$.
 \begin{claim}
Every map $A \in \mathrm{SL}(2,\RR)$ that fixes $a$ and $b$ globally fixes $K_{\theta_{0}}$.
\end{claim}
 In order to prove the claim, we introduce the maps $f^n_{\theta} \in
\mbox{SL}(2,\RR)$ defined by $f^n(\theta,x) = (\theta+n\omega,
f^n_{\theta}(x))$.  By definition of $\Theta_{1}$ there exists a point $z_{0}
= (\theta_0,x_0) \in \{\theta_{0}\} \times (a,b)$ with a dense positive orbit.
Let $z = (\theta_0,A(x_{0}))$, and consider an increasing sequence $(n_{i})$
such that $f^{n_{i}}(z_{0})$ converges to $z$.  By definition of $\Theta_{1}$,
the sequence $f^{n_{i}}_{\theta_{0}}(K_{\theta_{0}})$ converges to
$K_{\theta_{0}}$.  This implies that $f^{n_{i}}_{\theta_{0}}(a)$ converges to
$a$ and $f^{n_{i}}_{\theta_{0}}(b)$ converges to $b$.  The map $g \mapsto
(g(a),g(z_{0}),g(b))$ is a homeomorphism between $\mbox{SL}(2,\RR)$ and the
space of cyclically ordered triples.  Thus the sequence
$f^{n_{i}}_{\theta_{0}}$ converges to $A$, and
$A(K_{\theta_{0}})=K_{\theta_{0}}$, and the claim is proved.
 
Since $\mathrm{SL}(2,\RR)$ acts transitively on positively ordered triplets,
 this claim implies that $K_{\theta_{0}}$ either contains  $[b,a]$ or is disjoint from  $(b,a)$. In the first
 case $K_{\theta_{0}}$ would contain a point with dense positive orbit and $K$
 would equal $\TT^2$, a contradiction. Thus we have proved that for every
 $\theta_{0}\in \Theta_{1}$ the set $K_{\theta_{0}}$ contains either one or
 two points.
\end{proof}

\proof[Proof of proposition~\ref{p.sltr-onepoint}] First of all, suppose that
$f$ is semi-conjugate to an irrational rotation. Then it follows from the
classification for the dynamics of qpf linear circle homeomorphisms in
\cite{thieullen:1997} that the semi-conjugacy has linear fibre maps (i.e.\
$\pi_\theta \in \mbox{SL}(2,\R) \ \forall \theta \in\kreis$), which further
implies that the unique $f$-invariant measure is equivalent to the Lebesgue
measure. Hence $f$ is minimal.  Due to theorem~\ref{thm:poincare}, this means
that we can assume, without loss of generality, that $f$ is $\rho$-unbounded
and transitive.

 We argue by contradiction.  Applying proposition~\ref{p.twopoints}, we see
that $K_{\theta}$ contains exactly two points for a dense $G_{\delta }$ subset
$\Theta_{2}$ of the circle.  We can also assume that every point in
$\Theta_{2}$ is a continuity point of $\theta \mapsto K_{\theta}$.
 
 We fix some $\theta_{0} \in \Theta_{2}$. Up to a linear fibered conjugacy, we
 can assume that $K_{\theta_{0}} = \{1/8,5/8\}$. By continuity we can choose
 an open interval $I$ containing $\theta_{0}$ such that for each $\theta \in
 I$, $K_{\theta}$ does not meet $[1/4,1/2]$ neither $[3/4,1]$, and meets both
 $(0,1/4)$ and $(1/2,3/4)$.
 
We choose a lift $F$ of $f$ to $\kreis \times \RR$, and consider the lift
 $\tilde K$ of $K$.  By the choice of $I$, for every $z \in \tilde K \cap (I
 \times \RR)$ there exists $\tau(z) \in \frac{1}{2}\ZZ$ such that $z \in I
 \times (\tau(z),\tau(z)+1/4)$.  Now let $\theta,n$ be such that $\theta$ and
 $\theta+n\omega$ belong to $I \cap \Theta_{2}$.  The definitions of $I$ and
 $\Theta_{2}$ entails that, for a fixed value of $\theta$, the number
 $\tau(F^n(z))-\tau(z)$ is constant on $\tilde K_{\theta}$.  We denote this
 number by $d_{n}(\theta)$.

 \begin{claim}
There exist an interval $J \subset I$ and an integer $n$ such that $J+n\omega \subset I$
and the function $d_{n}$ is not constant on $J \cap \Theta_{2}$.
\end{claim}
 Let us prove the claim.  Let $J$ be any compact non-trivial interval inside
 $I$. There exists $n_{0}>0$ such that every interval $J'$ with the same
 length as $J$ has an iterate $J'+m\omega$ inside $I$ with $0 \leq m \leq
 n_{0}$.  Since $f$ is not $\rho$-bounded, the vertical diameter of the
 sequence of iterates $F^n(J \times \{0\})$ is not bounded. Thus there exists
 $N>0$ such that for every $m = 0, \dots, n_{0}$ the vertical diameter of
 $F^{N+m}(J \times \{0\})$ is greater than $1$.  By the choice of $n_{0}$ we
 can choose such an $m$ with $J+(N+m)\omega \subset I$.  Let $n=N+m$, then
 $d_{n}$ cannot be constant on $J \cap \Theta_{2}$ since this would imply that
 the vertical diameter of $F^{n}(J \times \{0\})$ is smaller than $3/4$. This
 proves the claim.

For every half integer $k$ let us consider the set $J_{k} =
\overline{d_{n}^{-1}\{k\}}$.  The interval $J$ is contained in the union of finitely
many $J_{k}$'s, thus the previous claim implies that there exist $k\neq k'$
and $\theta_{1} \in J \cap J_{k} \cap J_{k'}$.  As $\theta_{1}$ is a limit
point of $d_{n}^{-1}(k)$, there exist $x_{1} \in (0,1/4)$ and $x_{2} \in
(1/2,3/4)$ such that $F^n_{\theta_{1}}(x_{1}) \in (k,k+1/4)$ and
$F^n_{\theta_{1}}(x_{2}) \in (k+1/2,k+3/4)$.  Similarly there exist $x'_{1}
\in (0,1/4)$ and $x'_{2} \in (1/2,3/4)$ such that $F^n_{\theta_{1}}(x'_{1})
\in (k',k'+1/4)$ and $F^n_{\theta_{1}}(x'_{2}) \in (k'+1/2,k'+3/4)$.  To fix
ideas we assume that $x_{1} < x'_{1}$. Since $F^n_{\theta_{1}}$ is
order-preserving this implies that $k' > k$, and thus we also have $x_{2} <
x'_{2}$.  Since $x'_{1} < x_{2}$ we see that actually $k'=k+1/2$. Thus we have
$$
F^n_{\theta_{1}}((x_{1} x'_{1})) \supset \left[k+\frac{1}{4},k+\frac{1}{2}\right] \mbox{ and } 
F^n_{\theta_{1}}((x_{2} x'_{2})) \supset \left[k+\frac{3}{4},k+1\right].
$$ We consider the projective circle homeomorphism $\varphi:x \mapsto
f^n_{\theta_{1}}(x) - 1/4 -k$.  Both intervals $[0,1/4]$ and $[1/2,3/4]$ are
contained in the interior of their image by $\varphi$, and both intervals
$[1/4,1/2]$ and $[3/4,1]$ are contained in the interior of their pre-image:
thus $\varphi$ is not the identity but has a fixed point inside the interior
of each of these four intervals. This is absurd since $\varphi$ is a
projective map.
\qed

\proof[Proof of proposition~\ref{p.bjerkloev/johnson}]
Let $K$ be a minimal set. The following properties are well known.
\begin{claim}
There exist a compact invariant set $A \ssq \ntorus$ and $n\geq 1$ such that
\begin{itemize}
  \item $K \ssq A$;
  \item for all $\theta \in \kreis$ the set $A_\theta$ consists of exactly $n$ connected components;
  \item $K_\theta$ coincides with $A_\theta$ and contains exactly $n$ points
  whenever $\theta$ is a continuity point of the mapping $\theta \mapsto K_\theta$.
\end{itemize}
In particular, for all $\theta$ in a dense $G_\delta$ subset of $\kreis$,
the cardinality of $K_\theta$ is equal to $n$.
\end{claim}
Briefly spoken, this is due to the general fact that if $f$ has invariant strips,
then every minimal set is contained in an invariant strip, which
can furthermore be chosen minimal with respect to the inclusion amongst all
invariant strips. Such a minimal invariant strip automatically has the
required properties of the set $A$ above. 
In order to give precise references, we argue as follows. Recall that, by
assumption, $f$ has an invariant strip. It follows from \cite[Lemma 3.9]{jaeger/keller:2006} that
\[
 \rho(f) \ = \ \frac{k}{q}\omega + \frac{l}{p} \ \bmod 1
\]
for suitable integers $k,q,l,p \in \Z,\ q,p\neq 0$ (which are further
specified in \cite{jaeger/keller:2006}). If we go over to a suitable iterate,
consider a lift of $f$ to the $q$-fold cover $(\R / q\Z) \times \kreis$ and
perform a conjugacy of the form $\thx \mapsto (\theta,x-m\theta)$ with
suitable $m\in\Z$, then we can assume that $\rho(f) = 0$. Therefore, we can
choose a lift $F : \kreis \times \R \to \kreis \times \R$ with $\rho(F) =
0$. Let $\tilde K$ be the $\omega$-limit set of any point $\thx \in \kreis\times \RR$ which
lifts a point of $K$. Since $F$-orbits are bounded due to the bounded deviations,
$\tilde K$ is a minimal $F$-invariant set, which projects down to $K$. It
follows from the results in \cite{stark:2003}, for example corollary 4.4, that
the set
\[
\tilde A \ :=  \ \{ \thx \mid \inf \tilde K_\theta \leq x \leq \sup \tilde 
K_\theta \}
\]
is reduced to a point (and in particular coincides with $\tilde K$)
on all fibres which are continuity points of $\theta \mapsto \tilde K_\theta$.
The facts that $\tilde K \ssq \tilde A$ and that for
all $\theta \in \kreis $ the set $\tilde A_\theta$ consists of exactly one
connected component follows directly from the definition. By projecting
$\tilde A$ to the torus and redoing the transformations described above, we
obtain the required set $A$. This proves the claim.
\medskip

Now suppose $n>2$. We have to show that $K$ is is the graph of a continuous
$n$-valued curve in this case. If the cardinality of $K_\theta$ equals $n$ for any $\theta \in
\kreis$, then this follows easily. Hence, suppose for a contradiction that
there exists some $\theta_0\in\kreis$ with $\# K_{\theta_0} \geq
n+1$. Further, fix some $\theta_1\in \kreis$ which is a continuity point of
$\theta \mapsto K_\theta$. Note that this implies $\# K_{\theta_1} = c(K) =
n$, and let $K_{\theta_1} = \{ x_1 \ld x_n\}$.  Let $A \supseteq K$ be as
above and suppose $A_{\theta_0}^1 \ld A_{\theta_0}^n$ are the $n$ connected
components of $A_{\theta_0}$. Let $n_j$ be an increasing sequence of integers, such that
$\theta_0+n_j\omega \to \theta_1$. By going over to a subsequence and
relabelling if necessary, we can assume that each of the sequences
$f^{n_j}_{\theta_0}(A^i_{\theta_0})$ converges to the point $x_i$. (Note that
the continuity of $\theta \mapsto K_\theta$ in $\theta_1$ implies that the
limit points of different connected components of $A_{\theta_0}$ are
distinct.)
Let us choose arbitrary points $y_1\in A^1_{\theta_0}$, $y_2\in A^2_{\theta_0}$, $y_3\in A^3_{\theta_0}$.
As all $f_{\theta_0}^{n_j}$ are the projective actions of $\mathrm{SL}(2,\R)$-matrices,
this implies that the $f^{n_j}_{\theta_0}$ themselves converge to the linear circle homeomorphism
$g$ which maps $(y_1,y_2,y_3)$ on $(x_1,x_2,x_3)$.
However, this leads to a contradiction: at least one of the
intervals $A^i_{\theta_0}$ must be non-degenerate (since $K_{\theta_0} \ssq
A_{\theta_0}$ and $\#K_{\theta_0} \geq n+1$), but in the limit it is
contracted by $g$ to a single point $x_i$.
\qed

\end{document}